\documentclass[pdflatex,sn-mathphys-num]{sn-jnl}

\usepackage{amsthm}%

\usepackage{graphicx}%
\usepackage{multirow}%
\usepackage{amsmath,amssymb,amsfonts}%

\usepackage{mathrsfs}%
\usepackage[title]{appendix}%

\usepackage{textcomp}%
\usepackage{manyfoot}%
\usepackage{booktabs}%
\usepackage{algorithm}%
\usepackage{algorithmicx}%
\usepackage{algpseudocode}%
\usepackage{listings}%
\usepackage{float}

\newtheorem{theorem}{Theorem}%
\newtheorem{proposition}[theorem]{Proposition}%

\newtheorem{lemma}[theorem]{Lemma}
\newtheorem{conjecture}[theorem]{Conjecture}
\newtheorem{corollary}[theorem]{Corollary}

\newtheorem{example}{Example}%
\newtheorem{remark}{Remark}%

\newtheorem{definition}{Definition}%

\usepackage{graphicx}

\usepackage{caption,wrapfig}

\usepackage{subcaption}
\usepackage{mathtools}
\usepackage{amsmath}

\usepackage{lscape}
\usepackage{epstopdf}
\usepackage{mathrsfs}
\usepackage{tikz}
\usepackage{pgfplots}
\usepackage{enumitem}

\usepackage{amsmath,amstext,amsopn,amsfonts,eucal,amssymb}

\usepackage{graphicx,wrapfig,url}
\usepackage{color}
\usepackage{subcaption}
\usepackage{caption}

\def\rar{\rightarrow}
\def\mc{\mathcal}

\begin{document}
\title[Structures of Monoids Motivated by DNA Origami]{Structures of Monoids Motivated by DNA Origami}

\author
{\fnm{Peter} \sur{Alspaugh}}

\author
{\fnm{James} \sur{Garrett}}

\author
{\fnm{Nata\v sa} \sur{Jonoska}}

\author
{\fnm{Masahico} \sur{Saito}}

\affil
{\orgdiv{Department of Mathematics and Statistics}, \orgname{University of South Florida}, \orgaddress{\street{4202 E. Fowler Av.}, \city{Tampa}, \postcode{33620}, \state{FL}, \country{USA}}}

\abstract{We construct a class of monoids, called {\it origami monoids},  motivated by Jones monoids and by strand organization in DNA origami structures. Two types of basic building blocks of DNA origami closely associated with the graphical representation of Jones monoids are identified and are taken as generators for the origami monoid. Motivated by plausible modifications of the DNA origami structures and the relations of the well studied Jones monoids, we then identify a set of relations that characterize the origami monoid. These relations expand the relations of the Jones monoids and include a new set of relations called {\it contextual commutation}. With contextual commutation, certain generators commute only when found within a given context. 
	 We prove that the origami monoids are finite and  propose 
	 a normal form representation of their elements.
	 We establish a correspondence 
     between the Green's classes of the origami monoid and the Green's classes of a direct product of Jones monoids.}

\keywords{DNA origami,  Jones monoid, rewriting system, origami monoid, contextual commutation.}

\maketitle              

\section{Introduction}\label{intro-sec}

DNA origami, introduced by Rothemund \cite{rothemund-folding-2006} in 2006,
typically involves combining a
single-stranded cyclic viral molecule called a {\it scaffold}  with 200-250 short
{\it staple strands} to produce about  100{\it nm} diameter 2D shapes~\cite{rothemund-folding-2006}.
Fig.~\ref{origami_clean} shows a schematic of an origami structure,
where the thick black line
represents the scaffold, and the colored lines represent the staple strands that keep the scaffold folded in the shape.
Although success of DNA origami assemblies has been 
achieved experimentally (e.g.,~\cite{HPND11,Marchi2014,HQMW17,Bathe-hex}), 
theoretical understanding of strand organization within those shapes and characterizations of these structures are still lacking.
As algebraic systems that model DNA origami structures, monoids called {\it origami} monoids 
have been proposed in \cite{GarrettJKS19},  and  several alternative systems are suggested in \cite{NACO-Garrett2021}.
With this paper, we extend  results 
and prove some of the conjectures
stated in  \cite{GarrettJKS19}, thereby providing a better understanding of the origami monoids as algebraic structures.  
We prove the finiteness of the origami monoids, 
propose a normal form,
and study their 
Green's classes.

\begin{wrapfigure}{r}{0.42\textwidth}
 \vskip -.5cm
 \centering
 \includegraphics[scale=.3]{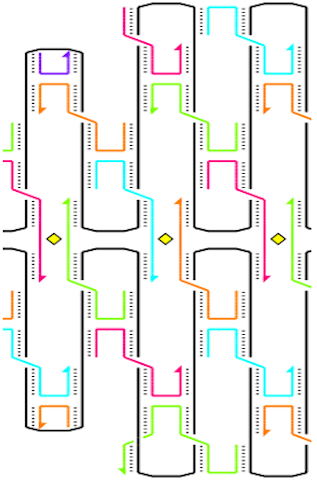}
 \captionsetup{width=\linewidth}
 \captionof{figure}{\small{A schematics  of DNA origami structure with scaffold in black and staples in color (edited from~\cite{rothemund-folding-2006}).}}.\label{origami_clean} \vskip-1cm
\end{wrapfigure}
More specifically, to model DNA origami structures,  we divide the general origami
structure (Fig.~\ref{origami_clean}) into small building blocks consisting of local scaffold-staple interactions,  and such building blocks are taken as   generators of an origami monoid.
 This approach is closely related to the diagrammatic description of generators in the
{\it Jones monoid} ~\cite{BDP,Jones83,East3} which is 
a monoid variant of the well studied Temperley-Lieb algebra~\cite{temperley}.  Temperley-Lieb algebras and Jones monoids have been  studied in physics and knot theory, in particular in relation to  the celebrated {\it Jones polynomial} of knots.
 The number of
generators
of an origami monoid ${\cal O}_n$ depends on the number $n$ of parallel folds of the scaffold in the DNA origami structure.
 The  correspondence between diagrams and monoid generators follows that of  the Jones monoid ${\cal J}_n$ of $n$ strings (Fig.~\ref{lieb1}). 
A DNA origami structure can be associated to  an element of an origami monoid, 
and 
a set of relations of the monoid 
 are derived in such a way that they are plausible for DNA segments to conform in DNA origami, as well as motivated from the set of  relations of Jones monoids.

\vskip -10pt

\begin{figure}[h]
	\centering
	\includegraphics[scale=0.45]{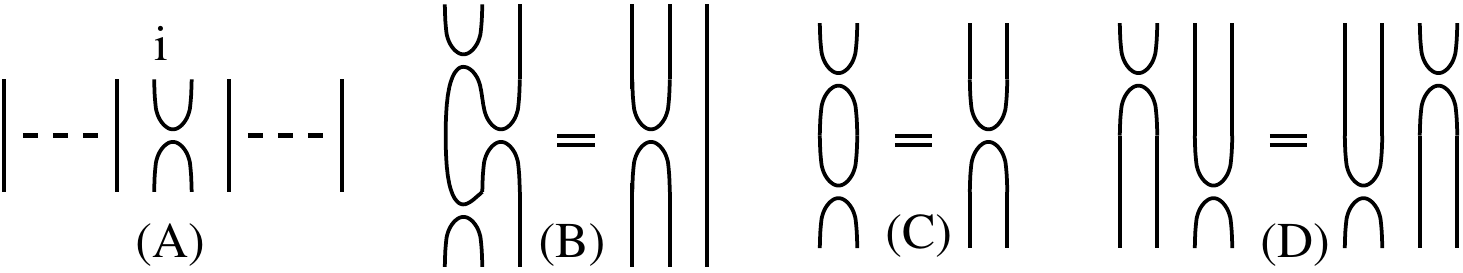}
	\caption{(A) The generators $h_i$ of $\mc J_n$ and relations (B,C,D) of the Jones monoid $\mc J_n$
		\cite{temperley}. 
	}
	\label{lieb1}\vskip 15pt
\end{figure}

\vskip -.5cm
Although the presentation of the origami monoids is motivated by DNA origami structures, our method of obtaining 
origami monoids
${\cal O}_n$  from
the Jones monoid ${\cal J}_n$
can be regarded as a general algebraic construction and can be applied to other algebraic structures.  
The set of generators is doubled into two types; to each generator $h_i$ (Fig.~\ref{lieb1} (A))
of ${\cal J}_n$, we assign two generators $\alpha_i$ and $\beta_i$ (Fig.~\ref{generators2}) of ${\cal O}_n$, $i=1, \ldots, n-1$. 
The original set of relations of the Jones monoid is imposed for each type of generators.
For example, for the idempotence rule $h_i^2=h_i$ for ${\cal J}_n$ (Fig.~\ref{lieb1} (C)), 
we impose idempotence on $\alpha_i$ and $\beta_i$, respectively. In addition, 
we impose idempotence on the products $\alpha_i \beta_i$, $\beta_i \alpha_i$, as well as the commutation of two types of generators with distinct indices reflecting the commutation in $\mc J_n$ (Fig.~\ref{lieb1}(D)). Algebraically,  this  definition can thus be seen as a construction of new monoids from old ones  by doubling  generators and expanding the relations
on all generators and products of pairs of respective generators. 

We propose a normal form for the origami monoids,
utilizing normal forms of the Jones monoids~\cite{BDP}.
For  $n=3, 4$, as computed by  a {\it GAP} program~\cite{GAP4}, we show that the number of elements according to our proposed normal form agrees with the origami monoid sizes.
In general, we prove that ${\cal O}_n$ is finite 
for all $n$.
The orders of ${\cal O}_n$
computed by ${\it GAP}$,
in fact, form an integer sequence that is not found in the {\it Online Encyclopedia of Integer Sequences (OEIS) }~\cite{OEIS}. 
Although the sizes of $\mc O_n$ are much larger than those of $\mc J_n$, we show a natural relationship between the Green's relations of an origami monoid and those of a direct product of Jones monoids~\cite{IdemP}.
In particular, we show that 
there is a bijective correspondence of Green's $\mc D$ classes between $\mc O_n$ and ${\cal J}_n \times {\cal J}_n$.

The paper is organized as follows. Background 
materials
are reviewed and notation is established in Section~\ref{sec:prelim}.
A definition of the origami monoids and motivations from DNA origami structures are explained in Section~\ref{sec:monoid}.
Two aspects of monoid structures of ${\cal O}_n$ are studied. A normal form 
representation for elements in $\mc O_n$ is proposed in Section~\ref{sec:normal}, while a bijective correspondence of Greens classes between 
${\cal O}_n$ and $ {\cal J}_n\times {\cal J}_n$  and $\mathcal{H}$-triviality of $\mathcal{O}_n$ are shown in Section~\ref{sec:Greens}.

\section{Preliminaries}\label{sec:prelim}

\subsection{Jones monoids}

The Jones monoid~\cite{IdemP,BDP,LF} $\mathcal J_n$ is
a monoid version of the Temperley-Lieb algebra \cite{temperley,Kauff}) that has been used in many fields, particularly in physics and  knot theory, 
and is defined with generators and relations as follows~\cite{BDP,LF}. 
The monoid $\mathcal J_n$  is generated by $h_i$, $i=1, \dots, n-1$, and has relations
$$
(B) \,\, h_ih_jh_i = h_i\ \text{ for } \ |i-j|=1,\
(C) \,\, (h_i)^2 =  h_i \ 
\
(D) \,\, h_ih_j =  h_jh_i \ \text{ for } \ |i-j|\geq2.
$$
The elements of $\mathcal J_n$ may be represented as planar diagrams with non-crossing lines connecting $n$ points at the top and $n$ points at the bottom of the diagrams. 
The diagram for the generator $h_i$ is depicted in Fig.~\ref{lieb1} (A) \cite{LF}. For each $h_i$, parallel vertical lines connect the top $j$th and bottom $j$th points ($j\not = i, i+1$) of the diagram for all but the $i$th and $(i{+}1)$st points, while the top $i$th and $(i{+}1)$st points are connected, and the bottom $i$th and $(i{+}1)$st points are connected. Multiplication of two elements is represented   by concatenation of diagrams by placing the diagram of the first element on top of the second and removing  closed loops. Diagramatic representations of the monoid  relations are depicted  in Fig.~\ref{lieb1} (B), (C) and (D). More details can be found in~\cite{BDP,LF}. An alternate presentation for $\mathcal{J}_n$ is found in \cite{East4}.

\subsection{String rewriting systems}\label{sec:string}

An alphabet $\Sigma$ is a non-empty, finite set of symbols. 
A word over $\Sigma$ is a finite sequence of elements (symbols) from $\Sigma$, and $\Sigma^*$ is the set of all words over $\Sigma$.
This set includes the empty string, the word containing no symbols, which is often written as $1$. A word $u$ is called a {\it factor of a word $v$} if there exist words $x$ and $y$, which may be empty, such that $v = xuy$. Note that this is also sometimes referred to as a subword.

A string rewriting system $(\Sigma,R)$ consists of an alphabet $\Sigma$ and a set of rewriting rules $R$ which is a binary relation on $\Sigma^*$. An element $(x,y)$ of $R$ is called a rewriting rule and is written $x \rightarrow y$.

If there is a sequence of words $u=x_1 \rightarrow x_2 \rightarrow \cdots \rightarrow x_n=v$ in a rewriting system
$(\Sigma^*, R)$, we will simply write $u \rightarrow v$.
An element $x \in \Sigma^*$  is {\it confluent} if for all $y, z \in \Sigma^*$ such that
$x \rightarrow y$ and $x \rightarrow z$, there exists $w \in \Sigma^*$
such that $y \rightarrow w$ and $z \rightarrow w$. If all words in $\Sigma^*$ 
are confluent,
then $(\Sigma^*, R)$ is called {\it confluent}. In particular, if $R$ is symmetric, then the system $(\Sigma^*, R)$ is confluent.

\subsection{Monoids and  Green's relations}

A monoid is a pair $(M,\cdot)$ where $M$ is a set and $\cdot$ is an 
associative binary operation on $M$ that has an identity element $1$. The set $\Sigma^*$
is a (free) monoid generated by  $\Sigma$ 
with word concatenation as the binary operation and the empty string as the identity element.
Presentations of monoids are defined from the free monoid in a manner similar to presentations of groups.
Rewriting systems define monoid relations by taking the equivalence closure of the rewriting rules, which makes the rewriting system confluent. 

For a monoid $M$, the {\it principal left (resp. right) ideal} generated by  $a \in M$  is defined by $Ma=\{ xa\ | \ x \in M\}$ (resp. $aM$), and the {\it principal two-sided ideal} is
$MaM$. Green's relations $\mathscr{L}$, $\mathscr{R}$, and $\mathscr{J}$ are defined for $a, b \in M$ by
$a \mathscr{L} b$ if $Ma=Mb$, $a \mathscr{R} b $ if $aM=bM$ and $a \mathscr{J} b$ if $MaM=MbM$.
Green's $\mathscr{H}$ relation is defined by $a \mathscr{H} b$ if $a \mathscr{L} b$ and $a \mathscr{R} b$. 
Green's $\mathscr{D}$ relation is defined by $a \mathscr{D} b$ if there is $c$ such that $a \mathscr{L} c $ and $c \mathscr{R} b$. The equivalence classes of $\mathscr{L}$ are called $\mathscr{L}$-classes, and similarly for the other relations. In a finite monoid, $\mathscr{D}$ and $\mathscr{J}$ coincide. 
The $\mathscr{D}$-classes can be represented in a matrix form called {\it egg boxes}, where the rows represent $\mathscr{R}$-classes, columns $\mathscr{L}$-classes, 
and each entry is 
an $\mathscr{H}$-class. 
See \cite{Pin} for more details.

\medskip
\begin{example}\label{ex:J3}
	{\rm
		\begin{sloppypar}
		In \cite{IdemP}, $\mathscr{D}$-classes are obtained for Jones and related monoids. $\mathcal{J}_3$ has two $\mathcal{D}$-classes. One contains only the identity element and the other is given by the $(2 \times 2)$-matrix below. Each cell in the matrix is a singleton $\mathscr H$-class containing only the indicated element.
		\[
		\left[
		\begin{array}{cc}
		h_1 & h_1 h_2 \\
		h_2 h_1 & h_2 
		\end{array}
		\right]
		\]
		The rows   $\{ h_1 , h_1 h_2 \} $, $\{ h_2 h_1, h_2 \}$, are the  $\mathscr{R}$-classes 
		and  columns  $\{ h_1 , h_2 h_1 \} $, $\{ h_1 h_2, h_2 \}$ are the $\mathscr{L}$-classes in this $\mathscr D$-class.  
		For instance, we see that multiplying $h_1$ and $h_2h_1$ by $h_i$ to the right 
		gives rise to the same right ideal. 
\end{sloppypar}
}
\end{example}

\section{Origami monoid $\mc O_n$}\label{sec:monoid}

In this section, we recall the definition of 
origami monoids
from \cite{GarrettJKS19}, and briefly relate them to 
  DNA origami structures.  
A more in depth discussion of the material in this section may be found in \cite{GarrettJKS19}.

\subsection{Generators}\label{sec:gen}

In DNA origami structures,  repeated patterns of 
simple building blocks whose concatenation builds  a larger structure can be observed, as in  Fig.~\ref{origami_clean}.
One type of these patterns is a cross-over by the staple strands, and the other is a cross-over of the scaffold strand.
Thus, a natural approach of describing DNA origami 
structures symbolically is to associate generators  of an algebraic system to 
simple building blocks  in which either a strand or staple crossing occurs, and to take multiplication in the system to be presented as 
concatenation of the blocks.

For a positive integer $n$, we define a monoid ${\mathcal O}_n$,  
where $n$ represents the number of vertical double stranded DNA strands, or equivalently, the number  of parallel folds of the scaffold. For the structure in Fig.~\ref{origami_clean}, observe that $n=6$. The generators of ${\mathcal O}_n$ are denoted by $\alpha_i$, which correspond to an antiparallel {\it{staple}} strand cross-over between strings $i$ and $i+1$ and $\beta_i$, which correspond to an antiparallel {\it{scaffold}} strand cross-over between strings $i$ and $i+1$ for $i=1, \ldots, n-1$,
as depicted in Fig.~\ref{generators2}.
The subscript  $i$ represents the position of the left scaffold corresponding to $\alpha_i$ and $\beta_i$, respectively, by starting at 1 from the leftmost scaffold strand fold and counting 
right (Fig.~\ref{context}).

\begin{figure}[h!]
    \begin{minipage}{.3\textwidth}
		\centering
		\includegraphics[scale=.2]{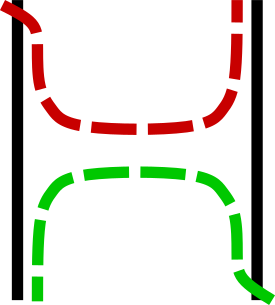}
		\subcaption{$\alpha_i$}
		\label{alpha2}
    \end{minipage}
	\begin{minipage}[h]{.3\textwidth}
		\centering
		\includegraphics[scale=.2]{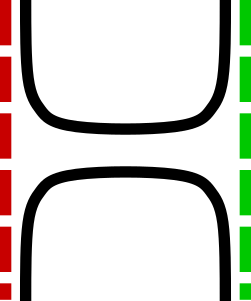}
		\subcaption{$\beta_i$}
		\label{beta2}
	\end{minipage}%
	\begin{minipage}{.35\textwidth}
	    \centering
	\includegraphics[scale=.2]{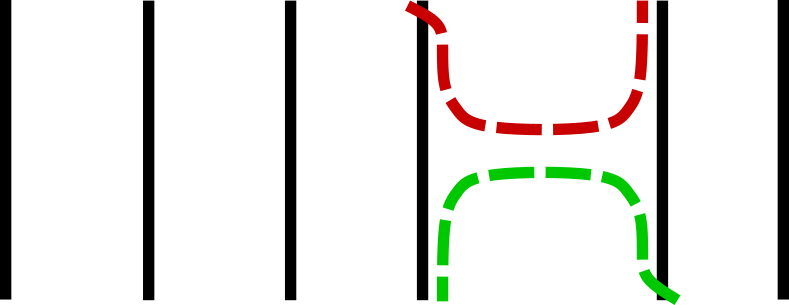}
	\subcaption{$\alpha_4$ in the context of a 6-fold stranded structure}
	\label{context}
	\end{minipage}
\caption{The generators identified}\label{generators2}	\vskip-.5cm
\end{figure}
Fig. \ref{context} shows a diagram corresponding to $\alpha_4$ as an example of the ``full picture'' of one of these generators. For the sake of brevity, 
we neglect to draw the extra scaffold and staple strands in most diagrams, but it may be helpful to imagine them when we describe their concatenation. In addition, we often use $\alpha_i$ and $\beta_i$ to refer to the corresponding diagrams.
In Fig.~\ref{context}, parallel scaffolds in generator diagrams are missing counterpart parallel staples.

\subsection{Concatenation as a monoid operation}\label{concat}

For a natural number $ n\geq2$, the set of generators of the monoid $\mc O_n$ is 
the set $\Sigma_{n}=\{\alpha_{1},\alpha_{2},\dots,\alpha_{n-1},\beta_{1},\beta_{2},\dots,\beta_{n-1}\}$.
For a product of two generators $x_i$ and $y_j$
in $\Sigma_n$,
we place the diagram of the  first generator above the second, lining up and connecting the scaffold strings of the two generators.
 If the two generators are adjacent, that is, if for indices $i$ and $j$ 
it holds that $|i-j|\leq1$, then we also connect their staples as described below. Otherwise, if $|i-j|\ge 2$, no staple connection is performed and the concatenation 
is finished.

We define a convention of connecting staples for adjacent generators, which is motivated by the manner in which
staples connect in Fig.~\ref{origami_clean}. 
Note how the staples of $\alpha$-type protrude ``outside" of the scaffold in the top-left and bottom-right corners of the diagram in Fig.~\ref{alpha2}. We refer to these ends of a staple as ``extending staple-ends'', and all other staple ends as ``non-extending staple-ends''. 
We connect staples everywhere \textit{except} when two non-extending staple-ends would have to cross a scaffold to connect (recall that the scaffold strands are connected first), as can be seen in Fig.~\ref{alphabeta} and Fig.~\ref{combinationgroup6}. As a result of this convention, $\alpha_i \beta_i$ and $\beta_i \alpha_i$ are two distinct words, i.e.,  $\alpha$'s and $\beta$'s do not freely commute.
Note that concatenation of three or more generators is associative because once the diagrams are placed on top of each other according to the concatenation order of the generators, the order of connecting the scaffold strands is not relevant, i.e., 
  can be done in an associative manner.

\begin{figure}[h!]
	\centering
	\begin{minipage}[h]{.33\textwidth}
		\centering
		\includegraphics[scale=.12]{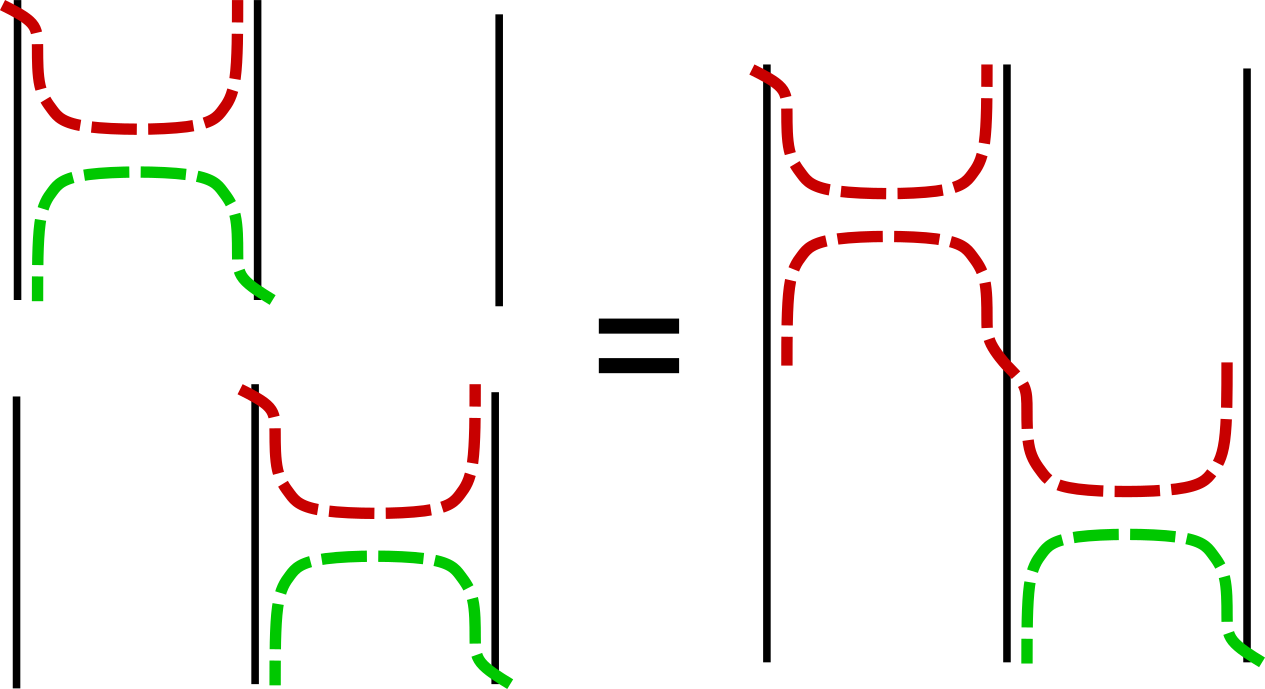}
		\subcaption{$\alpha_i\alpha_{i+1}$}
		\label{concatexample1}
	\end{minipage}%
	\begin{minipage}[h]{.33\textwidth}
		\centering
		\includegraphics[scale=.12]{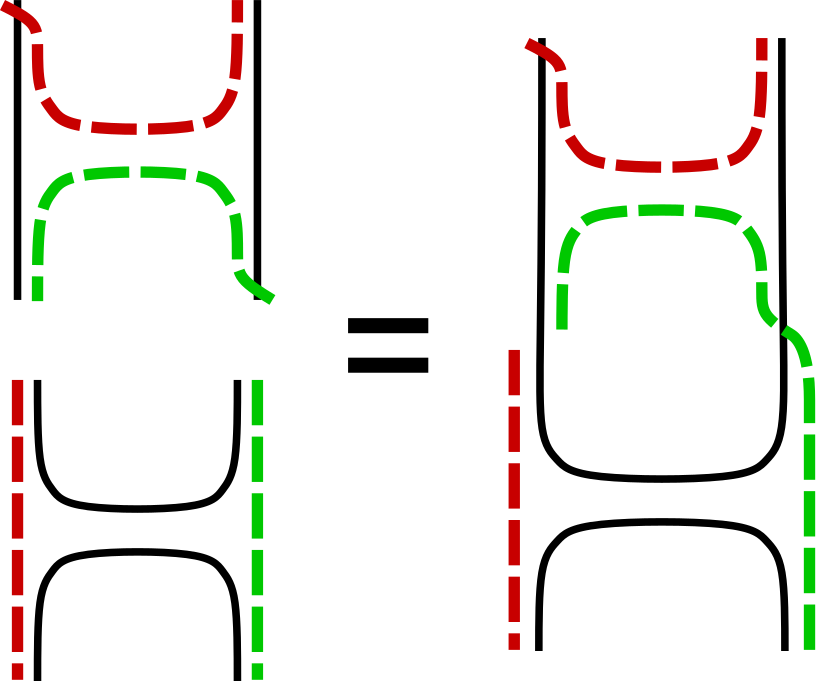}
		\subcaption{$\alpha_i\beta_i$}
		\label{alphabeta}
	\end{minipage}
	\begin{minipage}[h]{.33\textwidth}
		\centering
		\includegraphics[scale=.11]
  {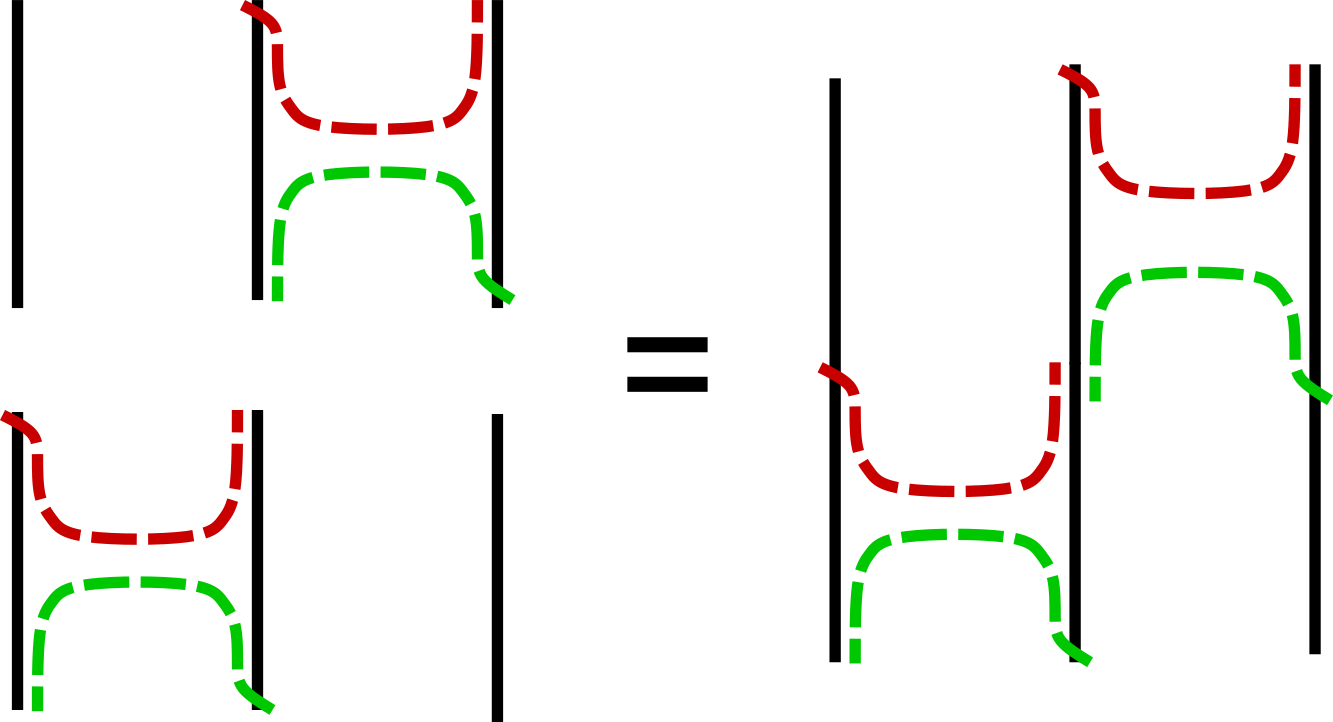}
		\subcaption{$\alpha_i\alpha_{i-1}$}
		\label{combinationgroup6}
	\end{minipage}
    \caption{Diagrams corresponding to products of origami monoid generators}
\end{figure}

\subsection{Relations in $\mc O_n$}\label{sec:writing}

\begin{sloppypar}
	The rewriting rules (which generate the relations within the monoids) 
are motivated by the similarity between the DNA origami structures (Fig.~\ref{origami_clean})
    and the diagrams of Jones 
    monoids (Fig.~\ref{lieb1}).
	It is deemed that the relations of Jones monoids simplify the DNA origami structure, and may be useful for designing efficient and more solid structures by the rewriting rules proposed below.
\end{sloppypar}

\subsubsection{Defining relations of $\mc O_n$}
 Given a set of generators
$\Sigma_{n}=\{\alpha_{1},\alpha_{2},\dots,\alpha_{n-1},\beta_{1},\beta_{2},\dots,\beta_{n-1}\}$, we  define a set of rewriting rules that simplify DNA origami structures similarly to how the relations of Jones monoids do to their respective diagrams. 
Using such simplification of diagrammatic representations,  in \cite{GarrettJKS19} we proposed  a set of rewriting rules on the alphabet $\Sigma_n$.
We define a string rewriting system $(\Sigma_{n},R)$ as follows.

 To ease the notation, we define an involution  on $\Sigma_{n}$, denoted by a bar, such that  $\overline{\alpha_i}=\beta_i$ and $\overline{\beta_i}=\alpha_i$, and  extend this operation to the free monoid $\Sigma^*_n$ by defining $\overline{w}$ for a word $w$ to be the word obtained from $w$ by applying a bar to each letter of $w$.
Let $\gamma \in \{\alpha,\beta\}$ and $i \in \{1,\dots,n-1\}$. Then we have:

$$
\begin{array}{llrlll}
{\rm (1)}\  & {\rm (Idempotence)} & (\gamma_{i})^2 & \rightarrow &\gamma_{i} & \\
{\rm (2)}\  & {\rm (Left \ Jones \ relation)} &\gamma_{i}\gamma_{i+1}\gamma_{i} & \rightarrow & \gamma_{i} & \\
{\rm (3)}\  & {\rm (Right\  Jones \ relation)} &\gamma_{i}\gamma_{i-1}\gamma_{i} & \rightarrow & \gamma_{i} & \\
{\rm (4)}\  & {\rm (Inter-commutation)} &\gamma_{i}\overline{\gamma_{j}} & \rightarrow & \overline{\gamma_{j}} \gamma_{i}, \ {\rm for} \ i\ne j \\
{\rm (5)}\  & {\rm (Intra-commutation)} &\gamma_{i}\gamma_{j} &\rightarrow & \gamma_{j}\gamma_{i}, \ {\rm for} \  |i-j|\geq2
\end{array}
$$

\noindent
The rules are extended to $\Sigma_{n}^*$ as described in Section~\ref{sec:string}.

Although the rewriting rules are inspired by Jones monoids, they also reflect the diagrams seen in DNA origami. In particular, rule $(4)$ is a consequence of the diagrams in Fig. \ref{rule3}, and is not related to Jones monoids. 
Moreover, this rule allows partial commutation between $\alpha$'s and $\beta$'s that results in a monoid $\mc O_n$ which is not merely a product of Jones monoids.

\begin{figure}[h!]
	\centering
		\includegraphics[scale=.1]{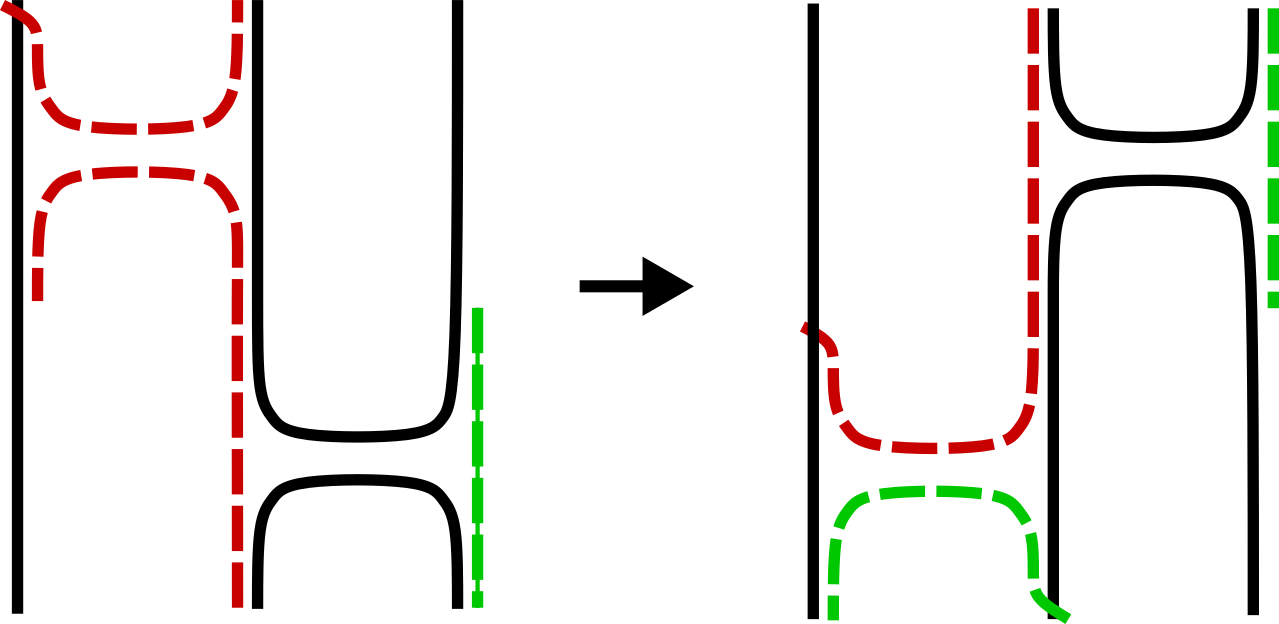}
		 \caption{{\small An example of the inter-commutation rewriting rule, $\alpha_i\beta_{i+1}\rightarrow \beta_{i+1}\alpha_i$, rule (4)}}
		\label{rule3}
\end{figure}

\subsubsection{Additional rewriting rules derived by substitution}\label{sec:rewriting2}

We note that a DNA origami structure has no internal loops, neither for the scaffold strand nor for the staples. Therefore, rules similar to $(1)$--$(5)$ can be applied to  products of $\alpha$'s and $\beta$'s. As shown in \cite{GarrettJKS19}, more  general substitution rules for the product of generators $\alpha_i$ and $\beta_i$ can be observed directly from the diagrams to extend the original rules to a larger set of rules. 
For example, we can use $\gamma_i=\alpha_i\beta_i$ in the idempotence rule (1) (see Fig.~\ref{substitution1})
and obtain a new rewriting rule that cannot be derived from the original set. 
These substitutions within the rules cannot always be justified by the diagrams. For example, substituting $\gamma_i=\alpha_i\beta_i$ in rule (4) produces inequivalent diagrams (Fig.~\ref{wrong_rule}). The expansion of rule (5) on $\gamma_i=\alpha_i\beta_i$ is a consequence of (4) and (5) so it does not become a new rule. We complete the set of rules and the definition of $\mc O_n$ with the following definition.

\begin{figure}[H]
     \begin{minipage}[h]{.28\textwidth}
		\centering
		\includegraphics[scale=.1]{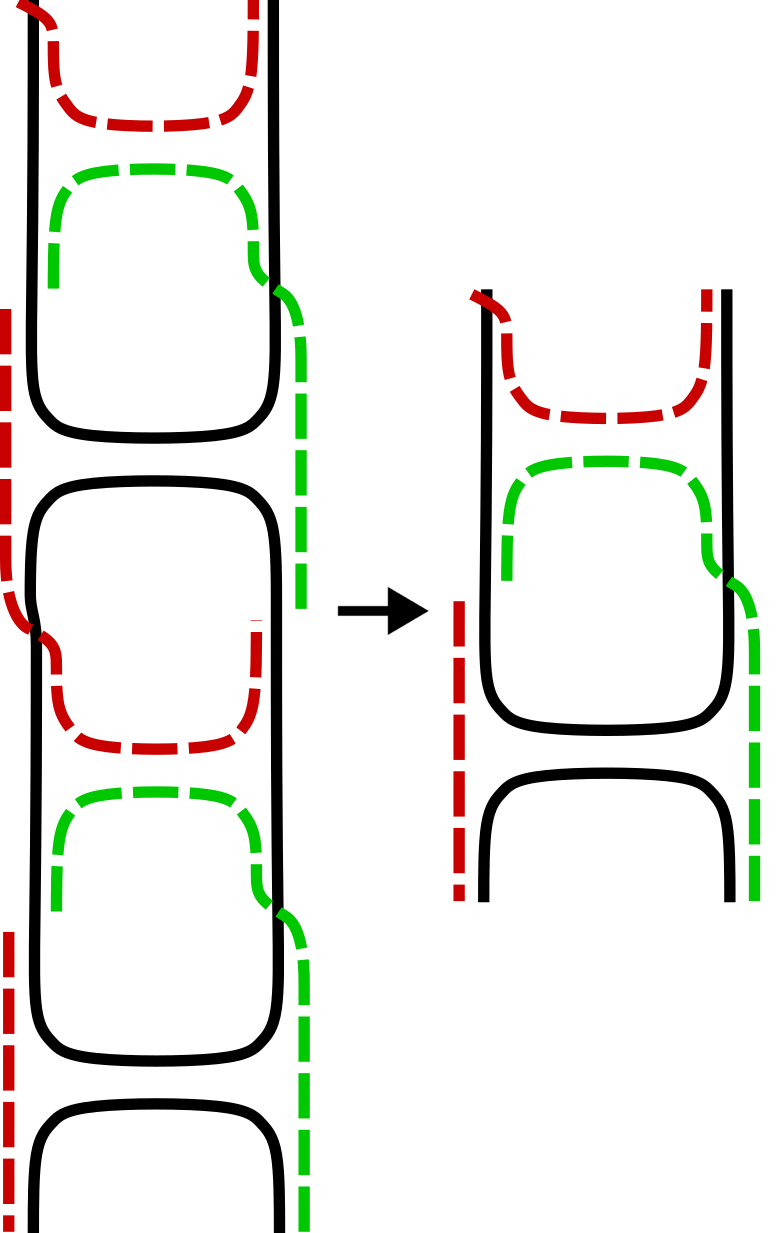}
\subcaption{Substitution of $\alpha\beta$ into the first rewriting rule: $(\alpha_i\beta_i)^2 \rightarrow \alpha_i\beta_i$}
		\label{abab}
	\end{minipage}\quad
	\begin{minipage}[h]{.28\textwidth}
		\centering
		\includegraphics[scale=.10]{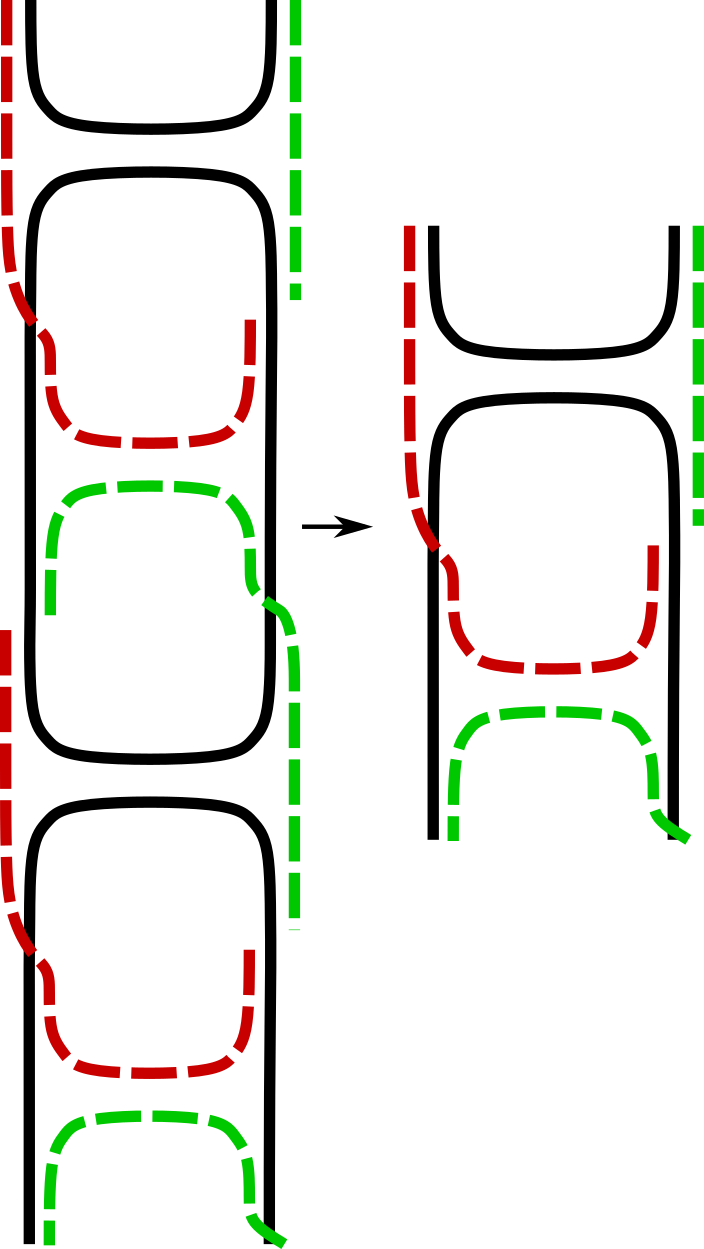}
		\subcaption{Substitution of  $\beta\alpha$ into the first rewriting rule: $(\beta_i\alpha_i)^2 \rightarrow \beta_i\alpha_i$}
		\label{baba}
	\end{minipage}
  \quad  
 \begin{minipage}[h]{.35\textwidth}
        \centering
	\includegraphics[scale=.09]{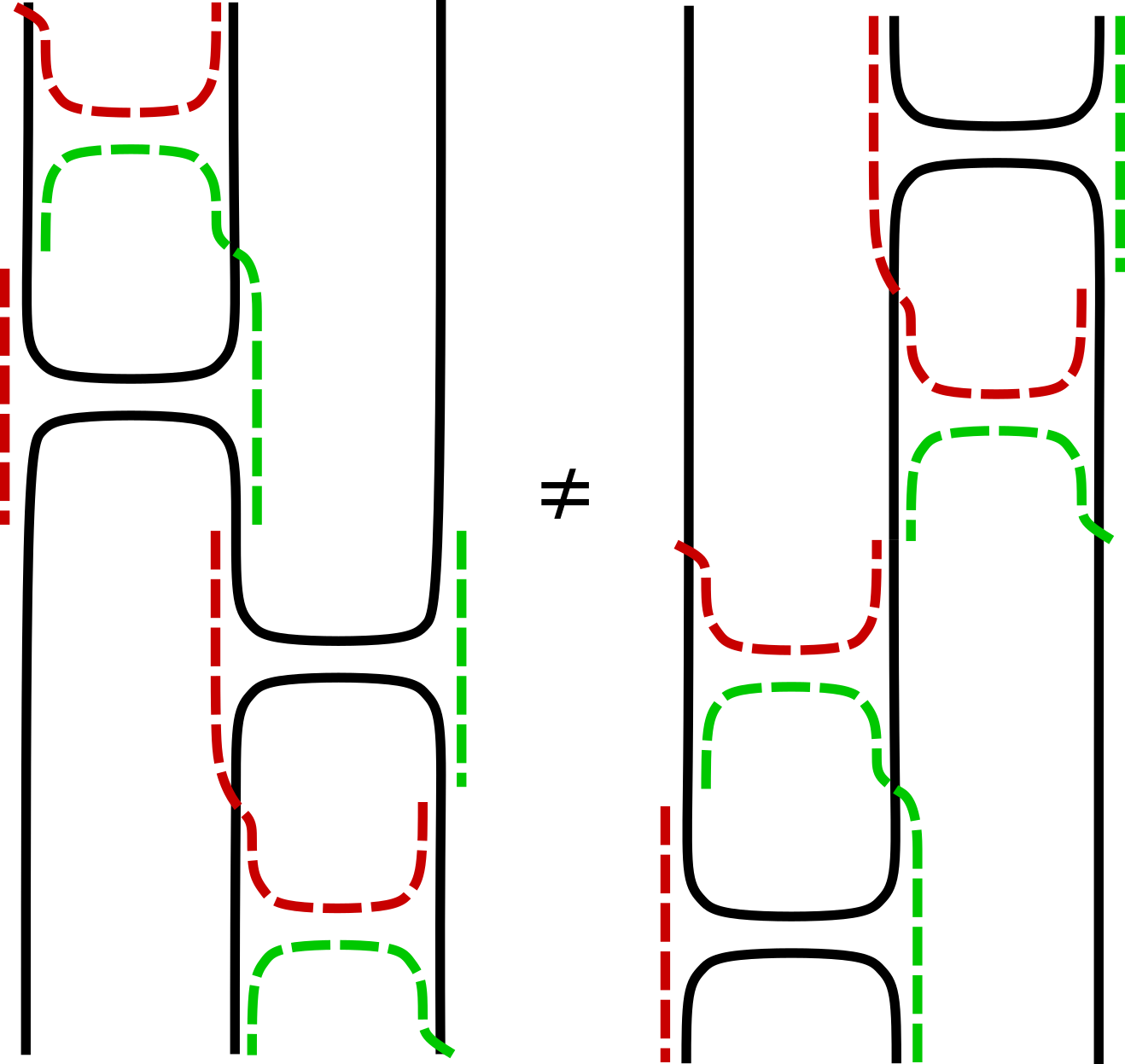}
	\subcaption{{\small Conflicting substitution of $\gamma=\alpha\beta$ into rewriting rule (4)
	}}
	\label{wrong_rule}
    \end{minipage}
    \caption{Substitution of $\gamma=\alpha\beta$ into rewriting rules (1) and (4). Substitution in (4) conflicts with the structure of the scaffold: the scaffold strand at the top left is connected to the second strand only on the left side of the figure, and on the right hand side of the figure it is connected to the third strand.}\label{substitution1}
\end{figure}

\medskip
\begin{definition}\label{origami-monoid}
	{\rm
		The \emph {origami monoid}  ${\mathcal O}_n$  is the monoid with a set of generators $\Sigma_n$  and relations generated by
		the rewriting rules (1) through (5), and for $\gamma \in \{\alpha, \beta \}$ and $i \in \{1, \dots, n-1\}$, the following additional rules:
		
		(1a) $(\gamma_i\overline{\gamma_i})^2 
		\rightarrow
		\gamma_i\overline{\gamma_i}$
		
		(2a) $\gamma_i\overline{\gamma_i}\gamma_{i+1}\overline{\gamma_{i+1}}\gamma_i\overline{\gamma_i} \rightarrow
		\gamma_i\overline{\gamma_i}$
		
		(3a) $\gamma_i\overline{\gamma_i}\gamma_{i-1}\overline{\gamma_{i-1}}\gamma_i\overline{\gamma_i} \rightarrow
		\gamma_i\overline{\gamma_i}$
		
		(2b) $\gamma_i\overline{\gamma_i}\gamma_i \gamma_{i+1}\overline{\gamma_{i+1}}\gamma_{i+1}\gamma_i\overline{\gamma_i}\gamma_i \rightarrow
		\gamma_i\overline{\gamma_i}\gamma_i   $
		
		(3b) $\gamma_i\overline{\gamma_i}\gamma_i \gamma_{i-1}\overline{\gamma_{i-1}}\gamma_{i-1}\gamma_i\overline{\gamma_i}\gamma_i \rightarrow 
		\gamma_i\overline{\gamma_i}\gamma_i   $
	}
\end{definition}
 
We focus on the algebraic study of the structures of origami monoids defined by the above presentations.

Based on GAP computations, we observe that the following rewriting rules hold in $\mathcal{O}_4$: 
$$
\alpha_3\alpha_1 \beta_1  \leftrightarrow\alpha_ 3 \beta_1 \alpha_1 , \quad  \alpha_1 \alpha_3 \beta_3 \leftrightarrow \alpha_1 \beta_3 \alpha_3, \quad  \beta_3\alpha_1 \beta_1  \leftrightarrow\beta_ 3 \beta_1 \alpha_1, \quad \beta_1 \alpha_3 \beta_3 \leftrightarrow \beta_1 \beta_3 \alpha_3,$$ $$\text{ and} \quad x\alpha_2\beta_2y \leftrightarrow x \beta_2 \alpha_2 y \text{, where } x,y \in \{ \alpha_1, \alpha_2, \alpha_3, \beta_1, \beta_2, \beta_3\}.  
$$ 
These computations motivated the proof that the following rewriting rules can be derived from those defining $\mc O_n$.

\medskip
\begin{remark}
    {\rm
    Note that all rewriting rules are symmetric, so $v\rightarrow w$ implies $w\rightarrow v$.
Because each of the original rewriting rules given in the monoid presentation is closed under the reverse operation,
given any new rewriting rule $v\rightarrow w$, $w^R\rightarrow v^R$ also holds. 
We specify when the reverse $w_1^R\to w_2^R$ of a rewriting rule $w_1\to w_2$ is used by writing \textit{rev.} next to the rule over the arrow denoting the rewriting.} 
    \end{remark}

\medskip
\begin{lemma}\label{lemma:identities}
Let $i,j \in \{1,2,\dots ,n-1\}$ such that $|i-j|=1$ and let $\gamma \in \{\alpha, \beta\}$. Then we have the following:
    \begin{itemize}
    \item[(a)]$\overline{\gamma_i}\gamma_i\overline{\gamma_j}\gamma_j\overline{\gamma_i}\rightarrow\overline{\gamma_i}\gamma_i\gamma_j$.\label{lem:ident-a}
    \item[(b)]$\overline{\gamma_i}\gamma_i\gamma_j\overline{\gamma_i}\rightarrow\overline{\gamma_i}\gamma_i\gamma_j$. \label{lem:ident-b}
    \item[(c)]
 $\overline{\gamma_i}\gamma_i\overline{\gamma_i}\gamma_j\rightarrow\overline{\gamma_i}\gamma_i\gamma_j$.
 \label{lem:ident-c}
    \end{itemize}
\end{lemma}

\begin{proof}
 (a)
    In the proof of this and the following lemmas, we underline the subword to be rewritten in each step and indicate the rewriting rule used 
    above the arrow. Note that rewriting rules (2) and (3) are always cited together because the lemma applies to both scenarios where $j=i+1$ or $j=i-1$. We compute
$$
\begin{array}{ccccccc}
\overline{\gamma_i}\gamma_i\overline{\gamma_j}\underline{\gamma_j\overline{\gamma_i}}
& 
\xrightarrow{(4)}
&
\overline{\gamma_i}\gamma_i\overline{\gamma_j}\overline{\gamma_i}\underline{\gamma_j}
&
\xrightarrow{(2)(3)}
&
\overline{\gamma_i}\gamma_i\overline{\gamma_j}\underline{\overline{\gamma_i}\gamma_j}\gamma_i\gamma_j 
&
\xrightarrow{(4)}
&
\underline{\overline{\gamma_i}\gamma_i\overline{\gamma_j}\gamma_j\overline{\gamma_i}\gamma_i}\gamma_j \\
\xrightarrow{(2a)(3a)}
& 
\overline{\gamma_i}\gamma_i\gamma_j . 
\end{array}
$$

(b)
    We compute
$$\underline{\overline{\gamma_i}\gamma_i\gamma_j}\overline{\gamma_i}
 \ 
 \xrightarrow{\text{Lem.~\ref{lemma:identities}(a)}}
 \ 
\overline{\gamma_i}\gamma_i\overline{\gamma_j}\gamma_j\underline{\overline{\gamma_i}\overline{\gamma_i}}
\ 
\xrightarrow{(1)} 
\ 
\underline{\overline{\gamma_i}\gamma_i\overline{\gamma_j}\gamma_j\overline{\gamma_i}}
\ 
\xrightarrow{\text{Lem.~\ref{lemma:identities}(a)}}
 \ 
 \overline{\gamma_i}\gamma_i\gamma_j .
$$
(c) this relation follows directly from (b) by observing that $\gamma_j\overline{\gamma_i}=\overline{\gamma_i}\gamma_j$ by rule (4).
\end{proof}

\begin{lemma}\label{lemma:general} Let $\mc O_n$ be an origami monoid where $n\geq 2$.
\begin{enumerate}
    \item[(i)] For any  $i \in \{1, \dots, n-1\}$ and $\gamma \in \{\alpha, \beta\}$ and for $|m-i|\ge 2$  
    we have that $\gamma_m\alpha_i \beta_i \rightarrow \gamma_m \beta_i \alpha_i$, or equivalently $\alpha_i \beta_i \gamma_m \rightarrow \beta_i \alpha_i \gamma_m$.\label{part1}
    \item[(ii)] For any $x,y \in \{\alpha_1, \dots, \alpha_{n-1}, \beta_1, \dots, \beta_{n-1}\}$, and any $i \in \{1, \dots, n-1\}$, we have that $x\alpha_i\beta_i y \rightarrow x \beta_i\alpha_i y$ holds in $\mathcal{O}_n$. \label{part2}
\end{enumerate}
\end{lemma}

\begin{proof} (i) We proceed by induction on $k=|m-i|\ge 2$. Let $k=2$ and 
$j$ be between $i$ and $m$, i.e.,  $|m-j|=1$. We compute 
$$
\begin{array}{cccccc}
\underline{\gamma_m\overline{\gamma_i}\gamma_i}
&
\xrightarrow{(4)(5)}
&
\overline{\gamma_i}\gamma_i\underline{\gamma_m}
&
\xrightarrow{(2)(3)}
&
\underline{\overline{\gamma_i}\gamma_i\gamma_m}\gamma_j\gamma_m
&
\xrightarrow{(4)(5)}
\\[2mm]
\gamma_m\underline{\overline{\gamma_i}\gamma_i\gamma_j}\gamma_m
&
\xrightarrow{\text{Lem.\ref{lem:ident-c}(c)}}
&
\underline{\gamma_m\overline{\gamma_i}\gamma_i\overline{\gamma_i}}\gamma_j\gamma_m
&
\xrightarrow{(4)(5)}
&
\overline{\gamma_i}\gamma_i\overline{\gamma_i}\underline{\gamma_m\gamma_j\gamma_m}
&
\xrightarrow{(2)(3)}
\\[2mm]
\underline{\overline{\gamma_i}\gamma_i\overline{\gamma_i}\gamma_m}
&
\xrightarrow{(4)(5)}
&
\underline{\gamma_m}\overline{\gamma_i}\gamma_i\overline{\gamma_i}
&
\xrightarrow{(2)(3)}
&
\gamma_m\gamma_j\underline{\gamma_m\overline{\gamma_i}\gamma_i\overline{\gamma_i}}
&
\xrightarrow{(4)(5)}
\\[2mm]
\gamma_m\underline{\gamma_j\overline{\gamma_i}\gamma_i\overline{\gamma_i}}\gamma_m
&
\xrightarrow{\text{rev. Lem.\ref{lem:ident-c}(c)}}
&
\gamma_m\gamma_j\underline{\gamma_i\overline{\gamma_i}\gamma_m}
&
\xrightarrow{(4)(5)}
&
\underline{\gamma_m\gamma_j\gamma_m}\gamma_i\overline{\gamma_i}
&
\xrightarrow{(2)(3)}
\\[2mm]
\gamma_m\gamma_i\overline{\gamma_i}. 
\end{array}
$$
Suppose the relation holds for $k$ and consider $|m-i|=k+1$. Again, let 
$j$ be between $i$ and $m$ with $|m-j|=1$. We have

$$
\begin{array}{cccccc}
\underline{\gamma_{m}}\gamma_i\overline{\gamma_i}
&
\xrightarrow{(2)(3)} 
&
\gamma_{m}\gamma_{j}\underline{\gamma_{m}}\gamma_i\overline{\gamma_i} 
&
\xrightarrow{(4)(5)} 
&
\gamma_{m}\gamma_{j}\underline{\gamma_i\overline{\gamma_i}}\gamma_{m} 
&
\xrightarrow{\text{ind. hyp.}}\\[2mm]
\gamma_{m}\gamma_{j}\overline{\gamma_i}\gamma_i\underline{\gamma_{m}}
&
\xrightarrow{(4)(5)}
&
\underline{\gamma_{m}\gamma_{j} \gamma_{m}}\overline{\gamma_i}\gamma_i
&
\xrightarrow{(2)(3)} 
&
\gamma_{m}\overline{\gamma_i}\gamma_i .
&
\end{array} 
$$

which completes the proof of the first part of the lemma.

\begin{table}[h]
\centering
\begin{tabular}{|c|c|c|c|}
\hline
$x=y=\gamma$ & $q=i+1$ & $q=i$ & $q=i-1$ \\
\hline
$p=i+1$ & (1A) $\gamma_{i+1}\overline{\gamma_i}\gamma_i\gamma_{i+1}\rightarrow$ &  (1B) $\gamma_{i+1}\overline{\gamma_i}\gamma_i\gamma_i\rightarrow$ & (1C) $\gamma_{i+1}\overline{\gamma_i}\gamma_i\gamma_{i-1}\rightarrow$ \\
 & $\gamma_{i+1}\gamma_i\overline{\gamma_i}\gamma_{i+1}$ & $\gamma_{i+1}\gamma_i\overline{\gamma_i}\gamma_i$ & $\gamma_{i+1}\gamma_i\overline{\gamma_i}\gamma_{i-1}$\\
\hline
$p=i$ & (2A) $\gamma_i\overline{\gamma_i}\gamma_i\gamma_{i+1}\rightarrow$ & (2B) $\gamma_i\overline{\gamma_i}\gamma_i\gamma_i\rightarrow$ & (2C) $\gamma_i\overline{\gamma_i}\gamma_i\gamma_{i-1}\rightarrow$ \\ 
& $\gamma_i\gamma_i\overline{\gamma_i}\gamma_{i+1}$ & $\gamma_i\gamma_i\overline{\gamma_i}\gamma_i$ & $\gamma_i\gamma_i\overline{\gamma_i}\gamma_{i-1}$\\
\hline
$p=i-1$ & (3A) $\gamma_{i-1}\overline{\gamma_i}\gamma_i\gamma_{i+1}\rightarrow$ & 
(3B) $\gamma_{i-1}\overline{\gamma_i}\gamma_i\gamma_i\rightarrow$ & 
(3C) $\gamma_{i-1}\overline{\gamma_i}\gamma_i\gamma_{i-1}\rightarrow$\\
& $\gamma_{i-1}\gamma_i\overline{\gamma_i}\gamma_{i+1}$ & $\gamma_{i-1}\gamma_i\overline{\gamma_i}\gamma_i$ & $\gamma_{i-1}\gamma_i\overline{\gamma_i}\gamma_{i-1}$\\
\hline
\end{tabular}
\caption{Cases where $x=y=\gamma$}\label{tab1}
\end{table}
\vskip -.5cm
(ii) 
Observe that when $x$ or $y$ has index $\leq i-2$ or $\geq i+2$, the statement follows directly from part (i). 
It then suffices to show that the statement holds in the  cases where $x$ and $y$ are both $\gamma$ or where $x$ is $\gamma$ and $y$ is $\overline{\gamma}$ (where $\gamma$ is one of $\alpha$ or $\beta$ and $\overline{\gamma}$ is the other) and where their indices (call them $p$ and $q$) differ from $i$ by less than 2, i.e., $p,q \in \{i-1, i, i+1\}$. This results in 18 cases (Tables~\ref{tab1} and ~\ref{tab2}). However, many of these cases are identical.
The tables show the cases where $x$ and $y$ are of the same type (Table~\ref{tab1}) and different (Table~\ref{tab2}).

\begin{table}[h!]
\centering
\begin{tabular}{|c|c|c|c|}
\hline
$x=\gamma$, $y=\overline{\gamma}$ & $q=i+1$ & $q=i$ & $q=i-1$ \\
\hline
$p=i+1$ & (1A) $\gamma_{i+1}\overline{\gamma_i}\gamma_i\overline{\gamma_{i+1}}\rightarrow$ &  (1B) $\gamma_{i+1}\overline{\gamma_i}\gamma_i\overline{\gamma_i}\rightarrow$ & (1C) $\gamma_{i+1}\overline{\gamma_i}\gamma_i\overline{\gamma_{i-1}}\rightarrow$ \\
& $\gamma_{i+1}\gamma_i\overline{\gamma_i}\overline{\gamma_{i+1}}$ & $\gamma_{i+1}\gamma_i\overline{\gamma_i}\overline{\gamma_i}$ & $\gamma_{i+1}\gamma_i\overline{\gamma_i}\overline{\gamma_{i-1}}$\\
\hline
$p=i$ & (2A) $\gamma_i\overline{\gamma_i}\gamma_i\overline{\gamma_{i+1}}\rightarrow$ & (2B) $\gamma_i\overline{\gamma_i}\gamma_i\overline{\gamma_i}\rightarrow$ & (2C) $\gamma_i\overline{\gamma_i}\gamma_i\overline{\gamma_{i-1}}\rightarrow$ \\
& $\gamma_i\gamma_i\overline{\gamma_i}\overline{\gamma_{i+1}}$ & $\gamma_i\gamma_i\overline{\gamma_i}\overline{\gamma_i}$ & $\gamma_i\gamma_i\overline{\gamma_i}\overline{\gamma_{i-1}}$\\
\hline
$p=i-1$ & (3A) $\gamma_{i-1}\overline{\gamma_i}\gamma_i\overline{\gamma_{i+1}}\rightarrow$ & 
(3B) $\gamma_{i-1}\overline{\gamma_i}\gamma_i\overline{\gamma_i}\rightarrow$ & 
(3C) $\gamma_{i-1}\overline{\gamma_i}\gamma_i\overline{\gamma_{i-1}}\rightarrow$\\
& $\gamma_{i-1}\gamma_i\overline{\gamma_i}\overline{\gamma_{i+1}}$ & $\gamma_{i-1}\gamma_i\overline{\gamma_i}\overline{\gamma_{i}}$ & $\gamma_{i-1}\gamma_i\overline{\gamma_i}\overline{\gamma_{i-1}}$\\
\hline
\end{tabular}
\caption{Cases where $x=\gamma$ and $y=\overline{\gamma}$}\label{tab2}
\end{table}

Notice that within each respective table, cells (1A) and (3C), (1B) and (3B), (1C) and (3A), and (2A) and (2C) are equivalent because reflection over cell (2B) simply represents switching $i+1$ and $i-1$ on the indices of $x$ and $y$, which causes no change in this context. Further, notice that cells (1B) and (2A) are also equivalent, as the words in each rule may be reversed and (in Table~\ref{tab2}) $\gamma$ and $\overline{\gamma}$ may be switched to obtain the other. Case (2B) holds by directly applying the idempotence rules (1) and (1a).  We then only have to consider cells (1A), (1B), (1C), on each table, leaving us with the following cases. Let $j=i\pm 1$ and $\ell=i\mp 1$ to differentiate between cases where $p=i\pm 1$ and $q=i\mp 1$.

\begin{center}
\begin{tabular}{lll}
 \qquad \ Tab.~\ref{tab1}& & Tab.~\ref{tab2}\\
(1A) \ $\gamma_j\overline{\gamma_i}\gamma_i\gamma_j\rightarrow\gamma_j\gamma_i\overline{\gamma_i}\gamma_j$
& &
 \ $\gamma_j\overline{\gamma_i}\gamma_i\overline{\gamma_j}\rightarrow\gamma_j\gamma_i\overline{\gamma_i}\overline{\gamma_j}$
\\
(1B) \
$\gamma_j\overline{\gamma_i}\gamma_i\gamma_i\rightarrow\gamma_j\gamma_i\overline{\gamma_i}\gamma_i$
& \quad &
 \ $\gamma_j\overline{\gamma_i}\gamma_i\overline{\gamma_i}\rightarrow\gamma_j\gamma_i\overline{\gamma_i}\overline{\gamma_i}$
\\
(1C)\ $\gamma_j\overline{\gamma_i}\gamma_i\gamma_\ell\rightarrow\gamma_j\gamma_i\overline{\gamma_i}\gamma_\ell$
& &
\ $\gamma_j\overline{\gamma_i}\gamma_i\overline{\gamma_\ell}\rightarrow\gamma_j\gamma_i\overline{\gamma_i}\overline{\gamma_\ell}$
\end{tabular}
\end{center}

\begin{itemize}

\item[(1A)]
We compute \\
$\underline{\gamma_j\overline{\gamma_i}}\gamma_i\gamma_j\xrightarrow{(4)(5)}\overline{\gamma_i}\underline{\gamma_j\gamma_i\gamma_j}\xrightarrow{(2)(3)}\underline{\overline{\gamma_i}\gamma_j}\xrightarrow{(4)(5)}\underline{\gamma_j}\overline{\gamma_i}\xrightarrow{(2)(3)}\gamma_j\gamma_i\underline{\gamma_j\overline{\gamma_i}}\xrightarrow{(4)(5)}\gamma_j\gamma_i\overline{\gamma_i}\gamma_j$.
\smallbreak
and

$\underline{\gamma_j}\overline{\gamma_i}\gamma_i\overline{\gamma_j}\xrightarrow{(2)(3)}\gamma_j\gamma_i\underline{\gamma_j\overline{\gamma_i}}\underline{\gamma_i\overline{\gamma_j}}\xrightarrow{(4)(5)}\gamma_j\underline{\gamma_i\overline{\gamma_i}\gamma_j\overline{\gamma_j}\gamma_i}\xrightarrow{\text{Lem.\ref{lem:ident-a}(a)}}\gamma_j\gamma_i\overline{\gamma_i}\overline{\gamma_j}$.
\smallbreak
\item[(1B)] We compute \\
$\gamma_j\overline{\gamma_i}\underline{\gamma_i\gamma_i}\xrightarrow{(1)}\gamma_j\underline{\overline{\gamma_i}\gamma_i}\xrightarrow{(1a)}\underline{\gamma_j\overline{\gamma_i}\gamma_i\overline{\gamma_i}}\gamma_i\xrightarrow{\text{rev. Lem.\ref{lem:ident-a}(c)}}\gamma_j\gamma_i\overline{\gamma_i}\gamma_i$.
\smallbreak
and

$\underline{\gamma_j\overline{\gamma_i}\gamma_i\overline{\gamma_i}}\xrightarrow{\text{rev. Lem.\ref{lem:ident-a}(c)}}\gamma_j\gamma_i\underline{\overline{\gamma_i}}\xrightarrow{(1)}\gamma_j\gamma_i\overline{\gamma_i}\overline{\gamma_i}$.

\smallbreak
\item[(1C)] We compute \\
$\gamma_j\underline{\overline{\gamma_i}\gamma_i\gamma_\ell}\xrightarrow{\text{Lem.\ref{lem:ident-a}(c)}}\underline{\gamma_j\overline{\gamma_i}\gamma_i\overline{\gamma_i}}\gamma_\ell\xrightarrow{\text{rev. Lem.\ref{lem:ident-a}(c)}}\gamma_j\gamma_i\overline{\gamma_i}\gamma_\ell$.
\smallbreak
and

$\gamma_j\underline{\overline{\gamma_i}\gamma_i}\overline{\gamma_\ell}\xrightarrow{(1a)}\underline{\gamma_j\overline{\gamma_i}\gamma_i\overline{\gamma_i}}\gamma_i\overline{\gamma_\ell}\xrightarrow{\text{rev. Lem.\ref{lem:ident-a}(c)}}\gamma_j\underline{\gamma_i\overline{\gamma_i}\gamma_i\overline{\gamma_\ell}}\xrightarrow{\text{Lem.\ref{lem:ident-a}(c)}}\gamma_j\gamma_i\overline{\gamma_i}\overline{\gamma_\ell}$.
\end{itemize}
\end{proof}

Lemma~\ref{lemma:general} implies that $\mathcal{O}_n$ exhibits \textit{contextual commutation} between generators $\alpha_i$ and $\beta_i$. That is, $\alpha_i$ and $\beta_i$ do not freely commute in general, but can when multiplied by a certain third element (in our case, any element with index $k$, where $|i-k|\geq 2$). In \cite{Contextual}, the authors explore what they call a \textit{contextual partial commutation} monoid, or \textit{c-trace} monoid, which is a generalization of a \textit{partial commutation} monoid, or \textit{trace} monoid. A \textit{c-trace monoid} is defined to be the quotient of a free monoid by relations $cab=cba$ for some generators $a,b,$ and $c$. Unlike c-trace relations, in $\mc O_n$, we find unexpected occurrences of contextual commutation not only between generators which are left or right multiplied by a third element, but also between two elements when multiplied by \textit{any} other elements on \textit{both} sides, as in the case of $x\alpha_i\beta_iy=x\beta_i\alpha_iy$.

\medskip
\begin{corollary}\label{cor:presentation}
    The origami monoid $\mc O_n$ has presentations defined with generators $\Sigma_n$ and relations generated with rules (1) -- (5), (1a), (2a) and (3a).
\end{corollary}

\begin{proof}
    In deriving the relations of contextual commutation given by Lemma~\ref{lemma:general}, none of the rules (2b) and (3b) were used. Therefore, we can apply contextual commutation and derive these two rules, i.e., these rules are  consequences of the other defining rules.
\end{proof}
\section{Towards normal forms}
\label{sec:normal}

 Using normal forms for representing elements of Jones monoids and the newly derived rules in Section~\ref{sec:writing}, 
we propose normal forms for elements in origami monoids.
 
Let ${\mathcal J}_n$ be the Jones monoid  
with generators $h_i$, $i=1, \ldots, n-1$.
We  denote the submonoid of $\mathcal{O}_n$ generated by $\alpha$'s
(resp. $\beta$'s),
by $\mathcal{O}^\alpha_n$ (resp. $\mathcal{O}^\beta_n$). An equivalent description for $\mathcal{O}^\alpha_n$ is the set of all words consisting of only $\alpha$'s (plus the empty word), and similarly for $\mathcal{O}^\beta_n$. Let $\mathcal{O}^{\alpha \beta}_n= [\mathcal{O}_n \setminus (\mathcal{O}^\alpha_n \cup \mathcal{O}^\beta_n)] \cup \{1\}$. 
In \cite{GarrettJKS19}, it was derived  that 
$\mathcal{O}^{\alpha \beta}_n$ is a submonoid of $\mathcal{O}_n$.
Using Lemma \ref{lemma:general}, we obtain a general form for any element of $\mathcal{O}_n$.

\medskip
\begin{corollary}\label{cor:form}
For every element $w\in \mc O_n$ there are $\gamma_1,\gamma_2\in \Sigma_n\cup\{1\}$ and $u\in \mc O^\alpha_n$, $v\in \mc O^\beta_n$ such that $w=\gamma_1 uv \gamma_2$. Moreover, if $i$ is the index of the first letter of $u$ and $j$ is the index of the last letter of $v$, then $\gamma_1 \in \{1, \beta_i\}$ and $\gamma_2 \in \{1, \alpha_j \}$.
\end{corollary}

\begin{proof}
   Let $w \in \mathcal{O}_n$. If $w=x\beta_k \alpha_m y$ for $x,y \in \mathcal{O}_n$ and $m \neq k$, then we can rewrite $w$ as $x \alpha_m \beta_k y$. In addition, by Lemma \ref{lemma:general}.(ii), if $x$ and $y$ are not empty, then if $w=x \beta_k \alpha_k y$, we can rewrite $w$ as $x \alpha_k \beta_k y$. Thus, except for the first and last letters, all the $\alpha$'s in $w$ can be moved to the beginning of the word and all $\beta$'s to the end of the word. That is, $w=\gamma_1 uv \gamma_2$, for $u, \gamma_2 \in \mathcal{O}_n^{\alpha}$, $v, \gamma_1 \in \mathcal{O}_n^{\beta}$, and $|\gamma_1|,|\gamma_2| \leq 1$. Let $i$ be the index of the first letter of $u$ and $j$ the index of the last letter of $v$.
   
   Let $\gamma_1 \not =1$. 
   If $\gamma_1 =\beta_k$, where $k \neq i$, then we can switch $\beta_k$ and $\alpha_i$, then move $\beta_k$ past all the $\alpha$'s to the end of the word (by the inter-commutation rule (4) and contextual commutation from Lemma~\ref{lemma:general}).  
   We then write $w$ as $u\beta_k v\gamma_2$, and we  consider $1$ as our new $\gamma_1$. Otherwise, we have $\gamma_1 = \beta_i$.  
   Similarly, by reversing the argument, we have that $\gamma_2 \in \{1, \alpha_j \}$.
\end{proof}

For $j \geq i$, define $h[j,i]=h_jh_{j-1}\cdots h_{i+1}h_i$. An element of $\mathcal{J}_n$, represented as a word over $\{h_1, \dots, h_{n-1}\}$, is said to be in \textit{Jones normal form} (J.n.f.) if it is of the form $h[j_1,i_1]\cdots h[j_k,i_k]$, where $j_1 < j_2 < \cdots < j_k$ and $i_1 < i_2 < \cdots < i_k$ \cite{BDP}. Define the normal forms for elements of $\mathcal{O}_n^{\alpha}$ and $\mathcal{O}_n^{\beta}$ to be the J.n.f. by replacing $h_i$ with $\alpha_i$ and $\beta_i$, respectively. We then have the following description to represent elements of $\mc O_n$.

\medskip
\begin{definition}\label{def:RegularForm}
    Let $w$ be an element of $\mathcal{O}_n$ represented as a word over $\Sigma_n$. Let $u$ be some word in $\mathcal{O}_n^{\alpha}$ which is in J.n.f. and let $v$ be the same in $\mathcal{O}_n^{\beta}$. Let $i$ be the index of the first letter of $u$ and $j$ be the index of the last letter of $v$. The $\textit{regular form}$ of $w$ is one of the following:
    \begin{enumerate}
        \item $uv$
        \item $\gamma_1uv$ or $uv\gamma_2$
        \item $\gamma_1uv\gamma_2$
    \end{enumerate}
    where in each case, $\gamma_1=\beta_i$ and $\gamma_2=\alpha_j$. In addition, the priorities of regular form are in the order listed. That is, if $w$ can be be written in form $(1)$, then $(1)$ is the regular form of $w$. Otherwise, if $w$ can be written in form $(2)$, then $(2)$ is the regular form of $w$. If $w$ can not be written in the form $(1)$ or $(2)$, then $(3)$ is the regular form of $w$. If $w$ can be written in the form $(2)$ in multiple  ways, then the shortest among them, if unique,  is the regular form of $w$. If the shortest is not unique, the form with  
    $\gamma_2=1$ is the regular form. The same holds if $w$ can be written in form $(3)$ in multiple ways.
\end{definition}

The definition above ensures that a regular form associated with an element in $\mc O_n$ is unique. Note that the only case when a word can be written in two ways of form $(2)$ of the same length is $\beta_i\alpha_i$ for some $i$, so we specify the regular form in this case. Lemmas \ref{lemma:normalform4} and \ref{lemma:normalform5} give examples of words being written in two ways of form $(2)$ with different lengths.
Our approach towards finding normal forms of ${\cal O}_n$  is to examine explicit words in a regular form.

\medskip
\begin{lemma}\label{lemma:normalform1}
  Let $w=\gamma_1uv\gamma_2$ be an element of $\mathcal{O}_n$ in a regular form with $\gamma_1 \neq 1$. If the J.n.f. 
  of $u$ is $\alpha[j_1,i_1]\cdots \alpha[j_k,i_k]$, then we must have that $j_{m+1}=j_m+1$ for $m=1, \dots, k-1$.
\end{lemma}

\begin{proof}
   We prove the contrapositive. Let $1 \leq m \leq k-1$ be the smallest $m$ such that $j_{m+1} > j_{m}+1$. Note $j_{m+1} - i_{m} \geq j_{m+1}-j_m \geq 2$, so we can rewrite $u$ as $$\alpha[j_1,i_1]\cdots\alpha_{j_{m+1}}\alpha[j_m,i_m]\alpha[j_{m+1}-1,i_{m+1}]\cdots \alpha[j_k,i_k].$$ But we also have $j_{m+1} - i_{m-1} \geq j_{m+1} -j_{m-1} > j_{m+1}=j_m \geq 2$, so we can repeat this process to rewrite $u$ as $\alpha_{j_1}\alpha_{j_{m+1}}\alpha[j_1-1,i_1]\cdots \alpha[j_{m+1}-1,i_{m+1}] \cdots \alpha[j_k,i_k]$. Finally, since $j_{m+1}-j_1 > j_{m+1}-j_{m} \geq 2$, by Lemma \ref{lemma:general}.(ii), we can switch the position of $\gamma_1$ and $\alpha_{j_1}$ in $w$. Then, again by Lemma \ref{lemma:general}.(ii), we may rewrite $w$ as $u\gamma_1v\gamma_2$. Thus the regular form of $w$ does not have $\gamma_1 \neq 1$. 
\end{proof}

\begin{figure}[h]
    \centering
\begin{tikzpicture}[scale=.8]
\begin{axis}[
xlabel=Position,
ylabel=Index,
xtick={0,...,8},
ytick={0,...,5}
]
\addplot[color=red] coordinates {
(1,3)
(2,2)
(3,1)
};
\addplot[color=red] coordinates {
(4,4)
(5,3)
(6,2)
};
\addplot[color=red] coordinates {
(7,5)
(8,4)
};

\addplot[blue, very thick, dotted] coordinates {
(1,3)
(4,4)
(7,5)
};
\end{axis}

\end{tikzpicture}
\caption{Graphical representation of index values of $u=\alpha_3\alpha_2\alpha_1\alpha_4\alpha_3\alpha_2\alpha_5\alpha_4$ in  Lemma \ref{lemma:normalform1}}
\label{plot1}
    \end{figure}
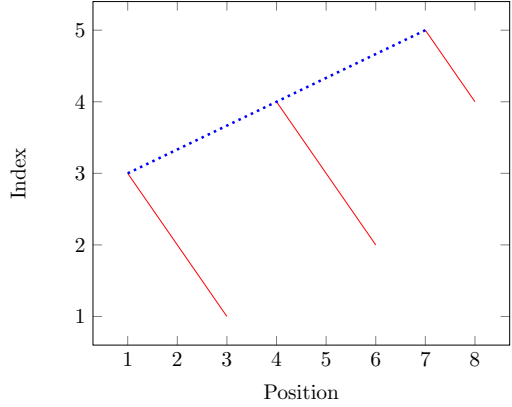
A graphical representation of index values of an example of a word 
$u$ described in Lemma~\ref{lemma:normalform1} is shown in Fig.~\ref{plot1}. The piecewise graph, in red, represents the values of the indices of 
$\alpha$ as a word 
$u$ in J.n.f. is scanned from left to right. Each red line constitutes one block 
$\alpha[j_s,i_s]$ in 
$u$. The dotted blue line shows the restriction on the first index in each block imposed by Lemma \ref{lemma:normalform1} such that the first index  
of each block increases by $1$ with every consecutive block. 
 By reversal of the argument in Lemma~\ref{lemma:normalform1}, we have the following corollary.

\medskip
\begin{corollary}\label{cor:normalform2}
 Let $w = \gamma_1 uv \gamma_2$ be an element of $\mathcal{O}_n$ in a regular form with $\gamma_2 \neq 1$. If the J.n.f.  
  of $v$ is $\beta[j_1,i_1]\cdots \beta[j_k,i_k]$, then we must have that $i_{m+1}=i_m+1$ for $m=1, \dots, k-1$.
\end{corollary}

In the following corollaries and lemmas, we fix $i$ to be the index of the first symbol in $u$ and $j$ to be the index of the last symbol in $v$.

\medskip
\begin{lemma}\label{lemma:normalform3}
  Let $w = \gamma_1 uv \gamma_2$ be an element of $\mathcal{O}_n$ in a regular form with 
  $\gamma_1=\beta_i$.  If the J.n.f.  
  of $v$ is $\beta[j_1,i_1]\cdots \beta[j_k,i_k]$, then we must have that $|j_1 -i| \leq 1$. Also, we must have $j_{m+1}=j_m+1$ for $m=2, \dots, k-1$, and either $j_{2}=j_1+1$ or $j_1=i-1$, $j_2=i+1$.
\end{lemma}

\begin{proof}
   We first consider the case when $|v| \geq 2$. Suppose $|j_1 -i| > 1$. Write $u=\alpha_i u'$ and $v=\beta_{j_1}v'$. Then we can rewrite $w$ as $\gamma_1\alpha_i \beta_{j_1}u'v' \gamma_2$ by Lemma \ref{lemma:general} since $v'$ is not empty. Then since $|j_1-i| > 1$, we can switch $\gamma_1$ and $\alpha_i$ and then move $\gamma_1$ to the right to write $w$ as $u\gamma_1v \gamma_2$.
   
   Now suppose that $j_{m+1} > j_m +1$ for some $m$, other than the case where $j_1=i-1$ and $j_2=i+1$. As in the proof of Lemma~\ref{lemma:normalform1}, we can rewrite $w$ as $\gamma_1 \alpha_i \beta_{j_{m+1}} u'v'\gamma_2$. Then since $j_{m+1} > j_m +1 \geq j_1 +1$ and since  $m>1$ and $|j_1-i|\le 1$, we have 
   $j_{m+1}=i+1$ and  $|j_{m+1}-i| > 1$. Again, as in the proof of Lemma~\ref{lemma:normalform1}, we can switch $\gamma_1$ and $\alpha_i$ and move $\gamma_1$ to the right to rewrite $w$ as $u\gamma_1v\gamma_2$.
   
   Now consider the case when $|v|=1$, and write $j=j_1$, so $v=\beta_{j_1}=\beta_j$. Assume that $|j-i| > 1$. Assume without loss of generality that $i < n-1$, so we can rewrite $w$ as $\beta_i \beta_{i+1}\beta_i \alpha_i u'\beta_j \gamma_2$ (the case of $i=n-1$ follows exactly the same by replacing $i+1$ with $i-1$). Write $u'=u''\alpha_k$. Then by Lemma $\ref{lemma:general}$.(ii) we can rewrite $w$ as $\beta_i\beta_{i+1}\alpha_i u'' \beta_i \alpha_k\beta_j \gamma_2$. If $j \neq k$, then we can switch $\alpha_k$ and $\beta_j$, and otherwise we have  $|k-i|=|j-i| > 1$, so by 
   Lemma~\ref{lemma:general}.(ii) we have $$w \leftrightarrow \beta_i\beta_{i+1}\alpha_i u'' \beta_i \beta_j \alpha_k \gamma_2 \leftrightarrow  \beta_i\beta_{i+1}\beta_i\alpha_i \beta_j u''  \alpha_k \gamma_2 \leftrightarrow \beta_i\alpha_i \beta_j u''  \alpha_k \gamma_2.$$ Again, we have $|j-1 | > 1$, so $$w \leftrightarrow \alpha_i\beta_i \beta_j u''  \alpha_k \gamma_2  \leftrightarrow  \alpha_i u''  \beta_i \beta_j\alpha_k \gamma_2  \leftrightarrow  \alpha_i u''  \alpha_k\beta_i \beta_j \gamma_2.$$
\end{proof}

 The above lemma also holds when we consider $u$ and $\gamma_2=\alpha_j$. That is, if the J.n.f.
 of $u$ is $\alpha[j_1,i_1]\cdots \alpha[j_k,i_k]$, then we have that $|i_k-j|\leq 1$ because all rewriting rules are closed under the reverse operation $^R$.

\medskip
\begin{lemma}\label{lemma:normalform4}
  Let $w = \gamma_1 uv \gamma_2$ be an element of $\mathcal{O}_n$ in a regular form with 
  $\gamma_1=\beta_i$ and $v=\beta_iv'$ for some $v' \in \mathcal{O}_n$.
  If the J.n.f. 
   of $u$ is $\alpha[j_1,i_1]\cdots \alpha[j_k,i_k]$, we must have that $i_k=i$,  $i_{m+1}=i_m+1$ for $m=1, \dots, k-1$, and that $v'=\gamma_2=1$.
\end{lemma}

\begin{proof}
   First suppose that $v'\neq 1$ or $\gamma_2\neq 1$. Then since $|v'\gamma_2| \geq 1$, by Lemma \ref{lemma:general}.(ii), we can rewrite $w$ as $\beta_i\beta_iuv' \gamma_2$, which can be rewritten as $\beta_i uv' \gamma_2$. Since $w$ can be rewritten as a shorter word, it was not in a regular form.
   
   Now consider the case when $|v\gamma_2|=1$, i.e., $v'=\gamma_2=1$. First, suppose that $i_k \neq i$. Writing $u=u'\alpha_{i_k}$, we then have that $w= \beta_i u' \alpha_{i_k}\beta_i \leftrightarrow \beta_i u' \beta_i \alpha_{i_k} \leftrightarrow \beta_i\beta_i u'  \alpha_{i_k}\leftrightarrow \beta_i u'  \alpha_{i_k}$, so $w$ was not in a regular form.
   
   Now assume that $i_k=i$ but $i_{\ell+1} > i_{\ell}+1$ for some $1 \leq \ell \leq k-1$. Let $m$ be the largest index such that $i_{m+1}> i_m +1 $. We have that $i_m+1 < i_{m+1} \leq j_{m+1}$, so we can switch the position of $\alpha_{i_{m}}$ and $\alpha_{j_{m+1}}$ within $w$. The same is true if we replace $\alpha_{j_{m+1}}$ with $\alpha_{j_{m+1}}-1$, so we can rewrite $w$ as 
   $$
   \beta_i\alpha[j_1,i_1]\cdots \alpha[j_m,i_m+1]\alpha[j_{m+1},i_{m+1}]\alpha_{i_m}\cdots \alpha[j_k,i_k]\beta_i.
   $$ 
   Then since $i_m+1 < i_{m+1} < i_{m+1}$, we can repeat this process to rewrite $w$ as 
   $$
   \beta_i\alpha[j_1,i_1]\cdots \alpha[j_m,i_m+1]\alpha[j_{m+1},i_{m+1}]\cdots \alpha[j_k,i_k]\alpha_{i_m}\beta_i.
   $$
   Then since $i_m+1 < i_{m+1} \leq i_k =i$, we can switch $\alpha_{i_m}$ and $\beta_i$. Then by Lemma \ref{lemma:general}.(ii), we can move $\beta_i$ to the left of all the $\alpha$'s to get $w \leftrightarrow \beta_i\beta_i u\leftrightarrow \beta_iu$, so $w$ was not in a regular form.
\end{proof}

\begin{lemma}\label{lemma:normalform5}
 Let $w = \gamma_1 uv \gamma_2$ be an element of $\mathcal{O}_n$ in a regular form with 
 $\gamma_1=\beta_i$, then we cannot have that $v=\beta_{i+1}\beta_{i}v'$ or $v=\beta_{i-1}\beta_iv'$ for some $v' \in \mathcal{O}_n$. 
\end{lemma}

\begin{proof}
   
   Suppose that $v=\beta_{i+1}\beta_iv'$ for some $v' \in \mathcal{O}_n$; the case when $v=\beta_{i-1}\beta_i$ is similar. Then, by applying rules (2) and (3) on the first $\beta_i$ and $\beta_{i+1}$, respectively, and then by applying rule (1), we have 
   $$
   w= \beta_iu\beta_{i+1}\beta_iv' \leftrightarrow \beta_i\beta_{i+1}\beta_i u \beta_{i+1}\beta_i\beta_{i+1}\beta_iv' \leftrightarrow  \beta_i\beta_{i+1}\beta_i \beta_i u \beta_{i+1}\beta_i\beta_{i+1}\beta_iv'.
   $$ 
   By Lemma \ref{lemma:general}.(ii), we obtain 
   $$
   w \leftrightarrow\beta_i\beta_{i+1}\beta_i  u \beta_i\beta_{i+1}\beta_i\beta_{i+1}\beta_iv' \leftrightarrow \beta_i u \beta_i\beta_{i+1}\beta_iv' \leftrightarrow\beta_i u\beta_iv'.
   $$
\end{proof}

We summarize the lemmas with the following two corollaries.

\medskip
\begin{corollary}\label{cor:forms}
Let $w=\gamma_1uv\gamma_2$ be an element of $\mathcal{O}_n$ in a regular form with $\gamma_1 \neq 1$. Let $\gamma_1=\beta_i$ and the J.n.f.
  of $v$ be $\beta[j_1,i_1]\cdots \beta[j_k,i_k]$. Then we have the following possibilities for $v$:
\begin{enumerate}
    \item $k=1$, $j_1=i_1=i$, i.e., $v=\beta_i$
    \item $k=1$, $j_1=i-1$
    \item $k\geq 2$, $j_1=i-1$, $j_2=i+1$
    \item $k\geq 1$, $j_1=i_1=i+1$, $j_m=i_m$ and $j_{m}=j_{m-1}+1$ for $m=2, \dots, k$
\end{enumerate}

In addition, in the case of $(1)$, if the J.n.f. 
 of $u$ is $\alpha[s_1,r_1]\cdots\alpha[s_t,r_t]$, we must also have that $r_t=i$, $r_{q+1}=r_q+1$ for $q=1, \dots, t-1$, and $\gamma_2=1$.
\end{corollary}

An example of a word $u$ in case (1) where $v=\beta_i$ for $i=3$ is depicted in Fig.~\ref{plot2}; a case where $v$ consists of a single block of $\beta$'s for $i=4$ is depicted (the first index of $\beta$ must be $3$) in Fig.~\ref{plot3}; and Fig.~\ref{plot4} shows an example where there are more than 2 blocks of $\beta$'s in $v$ and hence they have to start with index one less (i.e., $3$) and one more (i.e., $5$) than $i=4$, illustrating case (3). 
Case (4), where there are more than one $\beta$-blocks in $v$ with all being singletons for $i=2$, is illustrated in Fig.~\ref{plot5}.

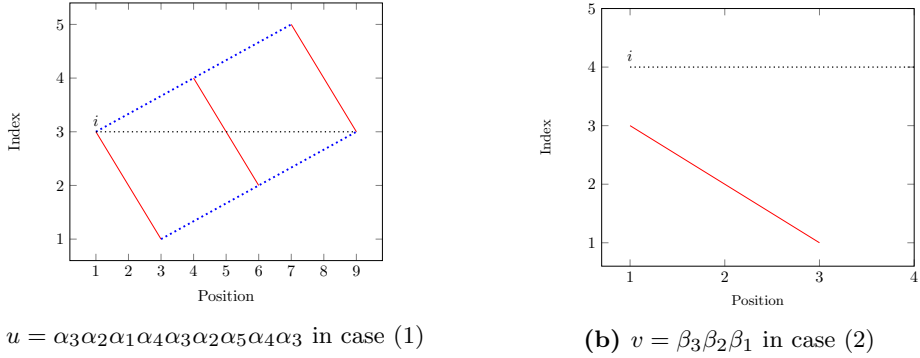
\begin{figure}[h]
    \begin{minipage}[h]{.47\textwidth}
    \centering
    \begin{tikzpicture}[scale=.6]
\begin{axis}[
xlabel=Position,
ylabel=Index,
xtick={0,...,9},
ytick={0,...,5}
]
\addplot[color=red] coordinates {
(1,3)
(2,2)
(3,1)
};
\addplot[color=red] coordinates {
(4,4)
(5,3)
(6,2)
};
\addplot[color=red] coordinates {
(7,5)
(8,4)
(9,3)
};

\addplot[blue, very thick, dotted] coordinates {
(1,3)
(4,4)
(7,5)
};
\addplot[blue, very thick, dotted] coordinates {
(3,1)
(6,2)
(9,3)
};

\addplot[black, thick, dotted] coordinates {
(1,3)
(9,3)
};

\node at (axis cs:1, 3.2){$i$};
\end{axis}
\end{tikzpicture}
\subcaption{$u=\alpha_3\alpha_2\alpha_1\alpha_4\alpha_3\alpha_2\alpha_5\alpha_4\alpha_3$ in case (1)}
\label{plot2}
    \end{minipage}\hfill
    \begin{minipage}[h]{.45\textwidth}
    \centering
\begin{tikzpicture}[scale=.6]
\begin{axis}[
xlabel={\small Position},
ylabel={\small Index},
xtick={0,...,4},
ytick={0,...,5},
xmax=4,
ymax=5
]
\addplot[color=red] coordinates {
(1,3)
(2,2)
(3,1)
};
\addplot[black, thick, dotted] coordinates {
(1,4)
(4,4)
};

\node at (axis cs:1, 4.2){$i$};
\end{axis}

\end{tikzpicture}
\subcaption{ $v=\beta_3\beta_2\beta_1$ in case (2)}
\label{plot3}
    \end{minipage}
    \caption{Examples of words corresponding to cases (1) and (2) of  Corollary \ref{cor:forms}}\label{fig:cases12}
\end{figure}

\medskip
\begin{corollary}\label{cor:forms2}

Let $w=\gamma_1uv\gamma_2$ be an element of $\mathcal{O}_n$ in a regular form with 
$\gamma_2=\alpha_j$ and the J.n.f. 
of $u$ be $\beta[j_1,i_1]\cdots \beta[j_k,i_k]$. Then we have the following possibilities for $u$:
\begin{enumerate}
    \item $k=1$, $j_1=i_1=i$, i.e., $u=\alpha_i$
    \item $k=1$, $j_1=i-1$
    \item $k\geq 2$, $i_{k}=i-1$, $i_{k-1}=i+1$
    \item $k\geq 1$, $j_k=i_k=i+1$, $j_m=i_m$ and $j_{m-1}=j_{m1}+1$ for $m=2, \dots, k$
\end{enumerate}

In addition, in the case of $(1)$, if the J.n.f. 
of $v$ is $\alpha[s_1,r_1]\cdots\alpha[s_t,r_t]$, we must also have that $r_t=i$, $r_{q+1}=r_q+1$ for $q=1, \dots, t-1$, and $\gamma_1=1$.
\end{corollary}

All these observations allow us to make the following conjecture about the normal forms of elements of $\mc O_n$.
\medskip
\begin{conjecture}\label{conj:normform}
Let $u \in \mathcal{O}_n^{\alpha}, v \in \mathcal{O}_n^{\beta}$ be in Jones normal form and let $i$ represent the first letter of $u$ and $j$ represent the last letter of $v$. A {\em normal form of an element of $\mathcal{O}_n$} is one of the following:

\begin{enumerate}
    \item $uv$
    \item $\beta_i uv$, and $u$ and $v$ are subject to restriction (a)
    \item $uv\alpha_j$, and $u$ and $v$ are subject to restriction (b)
    \item $\beta_i uv \alpha_j$, and $u$ and $v$ are subject to restrictions (a) and (b),
\end{enumerate} 
with the following restrictions:

\begin{enumerate}[label=\alph*)]
    \item $u$ is as in Lemma \ref{lemma:normalform1} and $v$ is one of the possibilities listed in Corollary \ref{cor:forms}
    \item $u$ is one of the possibilities listed in Corollary \ref{cor:forms2} and $v$ is as in Corollary \ref{cor:normalform2}
\end{enumerate}
\end{conjecture}
\medskip
By a normal form, we mean a form in which every element of $\mathcal{O}_n$ can be written uniquely. Corollary \ref{cor:form} provides a form in which every element can be written, but we cannot exlude the possibility that some element in $\mc O_n$ may have several such representations, so, we do not have unique representations with these forms. 
Note that while Definition \ref{def:RegularForm} does provide a unique regular form for each element in $\mathcal{O}_n$, it relies on choosing one regular form when multiple candidates are available, and is not descriptive, which is why we hesitate to call it a normal form. Conjecture \ref{conj:normform} describes a concrete description for a possible normal form of $\mathcal{O}_n$. However, we don't know whether 
the form in Conjecture \ref{conj:normform} is indeed unique.

We examine the regular forms for $n=2,3,4$.
It was proved in \cite{GarrettJKS19} that 
elements of ${\cal O}_2$ are
represented by 
	the following words:
	$\alpha_{1},\
	\beta_{1},\
	\alpha_{1}\beta_{1},\
	\beta_{1}\alpha_{1},\
	\alpha_{1}\beta_{1}\alpha_{1},\ {\rm  and}\
	\beta_{1}\alpha_{1}\beta_{1} $ which are in regular forms.
	For $n=3$ and $n=4$, the number of possible normal forms of words in $\mathcal{O}_n$ as described in Conjecture \ref{conj:normform} was manually calculated, compared to the size of $\mathcal{O}_n$, and was found to coincide. 
    For each of the cases $(1)-(4)$ in Conjecture \ref{conj:normform}, the number of $u$'s and $v$'s that meet the restrictions were counted and then multiplied to obtain the total number of combinations for that case. For example, in the case of $n=4$ in case $(2)$, we find that $v$ must be one of $\beta_1,\beta_2,\beta_3,\beta_1\beta_3,\beta_2\beta_1,\beta_2\beta_3,\beta_1\beta_3\beta_2$. If we look at a particular instance, such as $v=\beta_2\beta_1$, we then find that we must have $i=3$ and that $u$ is one of $\alpha_3,\alpha_3\alpha_2,\alpha_3\alpha_2\alpha_1$.

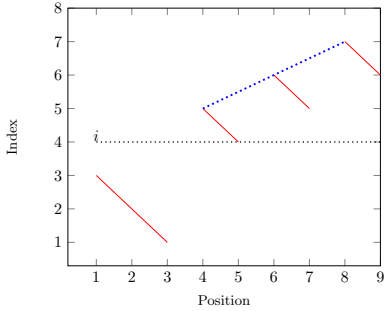
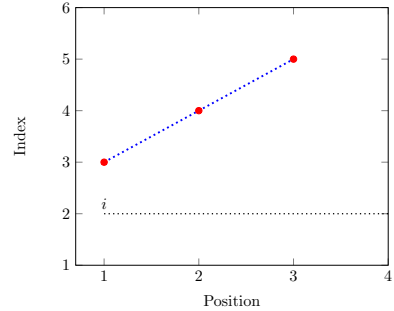
\begin{figure}[h]
    \begin{minipage}[h]{.45\textwidth}
    \centering
\begin{tikzpicture}[scale=.6]
\begin{axis}[
xlabel={\small Position},
ylabel={\small Index},
xtick={0,...,10},
ytick={0,...,9},
xmax=9,
ymax=8
]
\addplot[color=red] coordinates {
(1,3)
(2,2)
(3,1)
};
\addplot[color=red] coordinates {
(4,5)
(5,4)
};
\addplot[color=red] coordinates {
(6,6)
(7,5)
};

\addplot[color=red] coordinates {
(8,7)
(9,6)
};

\addplot[black, thick, dotted] coordinates {
(1,4)
(10,4)
};

\addplot[blue, very thick, dotted] coordinates {
(4,5)
(6,6)
(8,7)
};

\node at (axis cs:1, 4.2){$i$};
\end{axis}

\end{tikzpicture}
\subcaption{ $v=\beta_3\beta_2\beta_1\beta_5\beta_4\beta_6\beta_5\beta_7\beta_6$ in case (3) }
\label{plot4}
    \end{minipage}\hfill
   \begin{minipage}[h]{.45\textwidth}
   \centering
\begin{tikzpicture}[scale=.6]
\begin{axis}[
xlabel=Position,
ylabel=Index,
xtick={0,...,4},
ytick={0,...,6},
ymin=1,
xmax=4,
ymax=6
]
\addplot[mark=otimes*, color=red] coordinates {
(1,3)
};
\addplot[mark=otimes*, color=red] coordinates {
(2,4)
};
\addplot[mark=otimes*, color=red] coordinates {
(3,5)
};
\addplot[black, thick, dotted] coordinates {
(1,2)
(4,2)
};
\addplot[blue, very thick, dotted] coordinates {
(1,3)
(3,5)
};

\node at (axis cs:1, 2.2){$i$};
\end{axis}

\end{tikzpicture}
\subcaption{ $v=\beta_3\beta_4\beta_5$ in case (4) }
\label{plot5}
   \end{minipage}
 \caption{Examples of words corresponding to cases (3) and (4) of  Corollary \ref{cor:forms}}\label{fig:cases34}   
\end{figure}

It is known that the  elements of the Jones monoid ${\mathcal J}_n$
are in bijection with the linear chord diagrams obtained from the arcs of the diagrams representing them, and the total number  of such chord diagrams is equal to  the Catalan number
$\displaystyle C_n=\frac{1}{n+1}
\left( \begin{array}{cc} 2n \\ n \end{array} \right)$~\cite{BDP}.
Thus the numbers of  elements of ${\cal J}_n$ for $n=2, \ldots, 7$ are 
2, 5, 14, 42, 132, and 429 respectively.
GAP computations show that 
the number of non-identity elements in  ${\mathcal O}_3$, ${\mathcal O}_4$, ${\mathcal O}_5$, $\mathcal O_6$, and $\mathcal{O}_7$) are
44, 293, 2179, 19086, and 190512, respectively \cite{GAP4}. This sequence of integers is the new sequence A380196  in the OEIS~\cite{OEIS} list of sequences. 

\medskip
\begin{definition}
We call $p_{\alpha}:\mathcal{O}_n\to\mathcal{O}_n^{\alpha}$, defined by $p_{\alpha}(\alpha_i)=\alpha_i$ and $p_{\alpha}(\beta_i)=1$ for all $i\leq n-1$, the {\em $\alpha$-projection}. Define $p_{\beta}:\mathcal{O}_n\to\mathcal{O}_n^{\beta}$ as the {\em $\beta$-projection} in the same way for $\beta$'s. For $w\in\mathcal{O}_n$, we call the product $p_\alpha(w)p_\beta(w)$ the {\em core} of $w$. We will denote $p(w)=p_{\alpha}(w)p_{\beta}(w)$.
\end{definition}

\medskip
\begin{lemma}\label{lemma:finite}
 $\mathcal{O}_n$ is finite for all $n$.
\end{lemma}
\begin{proof}

  Let $u\in \mathcal{O}_n^{\alpha}$ and $v \in \mathcal{O}_n^{\beta}$ be non-empty words and let $i$ be the index of the first letter in $u$ and $j$ the index of the last letter in $v$. Now let $w \in \mathcal{O}_n$ with $p(w)=p(uv)$. We show that $w \in \{ uv, \beta_iuv, uv\alpha_j, \beta_iuv\alpha_j \}$, which gives $|\mathcal{O}_n| \leq 4 |\mathcal{J}_n|^2 < \infty$.
  
  Note by Lemma \ref{lemma:general} that for any word in $\mathcal{O}_n$, we may switch the order of $\alpha$'s and $\beta$'s freely within the word, except for the first and last letter of the word. Thus we may rewrite $w$ as $\beta_{k}u'v'\alpha_{\ell}$, where $u' \in \mathcal{O}_n^{\alpha}$, $v'\in \mathcal{O}_n^{\beta}$, and $u'\alpha_{\ell}=u$,$\beta_{k}v'=v$.
  Also note that we may switch $\alpha$'s and $\beta$'s at the beginning or end of a word as well, unless they have the same indices. Thus if $k\neq i$ and $\ell \neq j$, we may rewrite $w$ as $uv$. Assume that $k=i$ and $\ell=j$; the case where only one equality holds follows similarly. Then we may rewrite $w$ as $\beta_iu'v' \alpha_j\to\beta_i\beta_iu'v' \alpha_j\alpha_j\to\beta_iu'\alpha_j\beta_iv' \alpha_j\to\beta_iuv\alpha_j$.
\end{proof}

\section{Green's classes}\label{sec:Greens}

In this section, we explore the ideal structure within $\mc O_n$, in particular the relationship of the Green's classes between $\mathcal{O}_n$ and 
$\mathcal{J}_n$.

We recall some facts about the Green's classes of $\mathcal{J}_n$. 
From \cite{LF}, 
the Jones monoid is a regular $^*$-semigroup, which is
a semigroup with an additional unary operation $^*$ such that for all $v,w\in\mathcal{J}_n$, $w^{**}=w$, $(vw)^{*}=w^{*}v^{*}$, and $ww^{*}w=w$ \cite{*-semigroups}. In the case of the Jones monoid, since all relations 
are symmetric with respect to the reverse operation,  
$^*$ may be considered equivalent to the reverse operation $^R$. We say that the Jones monoid is a regular $^R$semigroup.
    
By \cite{LF}, $\mathcal{J}_n$ is $\mathcal{H}$-trivial, meaning any word in $\mathcal{J}_n$ is $\mathcal{H}$-related only to itself and that the Jones monoid has no nontrivial subgroups. It should also be stated that an equivalent condition for a monoid to be $\mathcal{H}$-trivial is that it is aperiodic, meaning that for all elements $a$ in the monoid, there exists a positive integer $n$ such that $a^n=a^{n-1}$ (Proposition 4.2 of \cite{Pin}).

For $v\in\mathcal{J}_n$, let {\sf cap}$(v)$ 
be the number of caps (cups) on the top (resp. bottom) of the diagram of $v$ (there are the same number). More formally, {\sf cap}$(v)$
is the number of edges on the subgraph of an element of $\mathcal{J}_n$ containing only the vertices on the top of the graph. Then $v\mathcal{D}w$ if and only if {\sf cap}$(v)=\,\, ${\sf cap}$(w)$
for $v, w \in\mathcal{J}_n$ \cite{LF}. This implies that the $\mathcal{D}$-classes of $\mathcal{J}_n$ are in bijection with the set consisting of $1$ and all elements of the form $h_1h_3\cdots h_{2k-1}$ for some $0<k\leq\lfloor\frac{n}{2}\rfloor$.

 Based on these facts about the Jones monoid, we now have several observations about the structure of $\mathcal{O}_n$ and its Green's classes.

\medskip
\begin{proposition}\label{*-semigroup}
$\mathcal{O}_n$ is a regular $^R$semigroup.
\end{proposition}

\begin{proof}
As per Definition~\ref{def:RegularForm}, we have three 
cases to consider
for the regular form of an element in $\mathcal{O}_n$.
\begin{enumerate}
\item $w=uv$
\item $w=\beta_iuv\alpha_j$
\item $w=\beta_iuv$ or
$w=uv\alpha_j$
\end{enumerate}
where $u \in \mathcal{O}_n^\alpha$, $v \in \mathcal{O}_n^\beta$, $i$ is the index of the first letter of $u$, and $j$ is the index of the last letter of $v$. In the case that $u=1$ or $v=1$, the regular form of a word either reduces directly to a word in $\mathcal{O}_n^{\alpha}$ or $\mathcal{O}_n^{\beta}$, or it reduces to a word $w=v'u'$ for some $v'\in\mathcal{O}_n^{\beta}$ and $u'\in\mathcal{O}_n^{\alpha}$, which is not meaningfully different from case (1). Thus, we assume that $|u|,|v|\geq 1$ in this proof unless otherwise stated.
\begin{enumerate}
\item Let $w \in \mathcal{O}_n$ have the regular form $uv$. We have that $w^R=v^Ru^R$. Then $ww^Rw=uvv^Ru^Ruv$, which reduces to $uu^Ruvv^Rv$ by contextual commutation of $\alpha$'s and $\beta$'s. This then simplifies to $uv$ because the Jones monoid is a
regular $^R$semigroup and $\mathcal{O}^\alpha_n\simeq\mathcal{O}^\beta_n\simeq\mathcal{J}_n$. 
\item Let $w$ have the regular form $\beta_iuv\alpha_j$. We have that 
$$
ww^Rw=\beta_iuv\alpha_j\alpha_jv^Ru^R\beta_i\beta_iuv\alpha_j . 
$$
This rewrites to $\beta_iuv\alpha_jv^Ru^R\beta_iuv\alpha_j$ from idempotence. By contextual commutation (Lemma~\ref{lemma:general}), it then rewrites to $\beta_iu\alpha_ju^Ruvv^R\beta_iv\alpha_j$. 
From Lemma \ref{lemma:normalform3}, we have that the index of the first letter of $v$ (call it $I$) and the last letter of $u$ (call it $J$) are such that $|i-I|\leq 1$ and $|j-J|\leq 1$. Notice that the first letter of $v$ is also the last letter of $v^R$ and vice versa for $u$. Therefore, by the Jones relations and idempotence, $\alpha_J\alpha_j\alpha_J\rightarrow\alpha_J\rightarrow\alpha_J\alpha_J$ and $\beta_I\beta_i\beta_I\rightarrow\beta_I\rightarrow\beta_I\beta_I$, so that
$\beta_iu\alpha_ju^Ruvv^R\beta_iv\alpha_j\rightarrow\beta_iuu^Ruvv^Rv\alpha_j$, which then rewrites to $\beta_iuv\alpha_j$ because $\mathcal{J}_n$ is a regular $^R$semigroup. If $v=1$, then $\beta_iu\alpha_j=\beta_iu'$, where $u'\in\mathcal{O}_n^{\alpha}$.
If $u=v=1$, the proof is an application of idempotence rules (1) and (1a).
\item This case is almost identical to that of part (2), except in the case where one of $u$ or $v$ is $1$ (let it be $v=1$ w.l.o.g.). Then we have either $w=u\alpha_j=u'\in\mathcal{O}_n^{\alpha}\simeq\mathcal{J}_n$, or $w=\beta_iu$. 
\end{enumerate}

We then have that $w^R$ 
satisfies $ww^R w= w$  
for $w$ in all cases. Trivially, also notice that $^R$ satisfies 
$w^{RR}=w$ and $(w_1w_2)^R=w_2^R w_1^R$, and hence  
$\mathcal{O}_n$ is a regular $^R$semigroup.
\end{proof}

\begin{proposition}
$\mathcal{O}_n$ is $\mathcal{H}$-trivial.
\end{proposition}

\begin{proof}
Let $w\in\mathcal{O}_n$. Similarly to Proposition \ref{*-semigroup}, $w$ has three
possible regular forms, but we shall examine only the case where $w=\beta_iuv\alpha_j$, as other cases are similar. In cases where $u$ or $v$ are $1$, the proof follows directly from contextual commutation (Lemma~\ref{lemma:general}).

Consider $u\alpha_j\in\mathcal{O}_n^\alpha\simeq\mathcal{J}_n$ and $v\beta_i\in\mathcal{O}_n^\beta\simeq\mathcal{J}_n$. Because $\mathcal{J}_n$ is aperiodic, there exists some $m\geq 2$ such that $x^m=w^{m-1}$ for all $x\in\mathcal{J}_n$. Since $\mathcal{O}_n^{\alpha}\simeq\mathcal{O}_n^{\beta}\simeq\mathcal{J}_n$, there exist $m_1,m_2\in\mathbb{Z}^+$ such that $(u\alpha_j)^{m_1}=(u\alpha_j)^{m_1-1}$ and $(v\beta_i)^{m_2}=(v\beta_i)^{m_2-1}$. Choose $m>max\{m_1,m_2\}$. Let $w=\beta_iuv\alpha_j$ be in a regular form. Then $w^m=(\beta_iuv\alpha_j)^m=\beta_iuv\alpha_j\beta_iuv\alpha_j\cdots\beta_iuv\alpha_j$. From contextual commutation in Lemma \ref{lemma:general}, we have that $xv\alpha_iy=x\alpha_ivy$ for any $x,y\in\mathcal{O}_n$, therefore $w^m=(\beta_iuv\alpha_j)^m=
\beta_i u(\alpha_jv\beta_iu)^{m-1}v\alpha_j$.  
 Again from contextual commutation (Lemma \ref{lemma:general}), $xv\beta_iu\alpha_jy=xu\alpha_jv\beta_iy$ whenever $x,y\not = 1$.  
 Thus we have $w^m=\beta_i u (\alpha_ju)^{m-1}(\beta_iv)^{m-1}v\alpha_j$.
 Because $m>max\{m_1,m_2\}$, $w^m=\beta_i u (\alpha_ju)^{m-2}(\beta_iv)^{m-2}v\alpha_j=w^{m-1}$. 
 
  Since for all $w\in\mathcal{O}_n$ we have $w^m=w^{m-1}$ for some $m$,  $\mathcal{O}_n$ is aperiodic and thus $\mathcal{H}$-trivial by Proposition 4.2 in \cite{Pin}.
\end{proof}

\medskip
\begin{lemma}\label{core-d-related}
For every $w\in \mc O_n$, $w \,\mc D\, p(w)$.
\end{lemma}

\begin{proof}
Let $w\in \mc O_n$ have a regular form $w=\gamma_iuv{\gamma_j}$  
where $\gamma_i\in\{1,\beta_i\}$ and ${\gamma_j}\in\{1,\alpha_j\}$, $i$ is the index of the first letter of $u$ and $j$ is the index of the last letter of $v$ (provided $u,v\neq 1$). We assume $u,v\neq 1$, as if one or both of $u=1$ or $v=1$ is true, then $w=v'u'$ for some $v'\in\mathcal{O}_n^{\beta}$ and $u'\in\mathcal{O}_n^{\alpha}$, which will be treated later.
Then $u=\alpha_i u'$ and $v=v'\beta_j$ and $p(w)=u\gamma_2\gamma_1v$. Using contextual commutation and idempotence, we may rewrite $\alpha_i w \beta_j=\alpha_i \gamma_1 uv \gamma_2 \beta_j\to \alpha_i \alpha_i u'\gamma_1 \gamma_2 v'\beta_j \beta_j\to\alpha_iu'\gamma_2\gamma_1v'\beta_j\to u\gamma_2\gamma_1v= p(w)$.
On the other side, $\gamma_1p(w)\gamma_2=\gamma_1u\gamma_2\gamma_1v\gamma_2\to\gamma_1\gamma_1uv\gamma_2\gamma_2= w$. Hence, $w \,\mc D\, p_\alpha(w)p_\beta(w)$.

In the case when $u=1$ or $v=1$ or both, then $w=z_1z_2$ where $z_1\in \mc O_n^\beta $ and $z_2\in \mc O_n^\alpha$ such that $z_1,z_2\neq 1$.  
Suppose $z_1=z_1'\beta_j$ and 
$z_2=\alpha_i z_2'$. We have $p(w)=z_2z_1$.
Then $\alpha_i w \beta_j= \alpha_i z_1 z_2 \beta_j\to\alpha_i z_2 z_1 \beta_j \to
\alpha_i\alpha_i z_2' z_1'\beta_j\beta_j\to \alpha_iz_2z_1\beta_j \to z_2z_1=p(w)$.
Similarly, if $\beta_I$ is the first symbol of $z_1$ and $\alpha_J$ is the last symbol of $z_2$, we have  $\beta_I p(w)\alpha_J=\beta_I z_2 z_1 \alpha_J=\beta_I z_1 z_2 \alpha_J =z_1z_2 $. Hence, again $w \,\mc D\, p_\alpha(w)p_\beta(w)=p(w)$.
\end{proof}

Furthermore, one can see that $\beta_iuv$ is $\mathcal{L}$-related to $uv$ when $v\neq 1$ and $uv\alpha_j$ is $\mathcal{R}$-related to $uv$ when $u\neq 1$ in a similar manner to above. To illustrate, by Lemma~\ref{lemma:normalform3}, if $I$ is the index of the first letter of $v$, then $|i-I|\leq 1$. Let $a_J$ be the first letter of $u$. Then $\alpha_J\beta_I(\beta_iuv)=\alpha_Ju\beta_I\beta_iv=uv$ from idempotence, the Jones relations (2)(3), and contextual commutation (Lemma~\ref{lemma:general}). 

\medskip
\begin{theorem}\label{thm:main}
Denote $\mathcal{D}_w$ to be the $\mathcal{D}$-class of the word $w$ and suppose $p(w)=uv$ for $u\in \mc O_n^\alpha$ and $v\in \mc O_n^\beta$. 
Then $z\in \mathcal{D}_{uv}$ if and only if $p_\alpha(z)\in \mc D_u$ and $p_\beta(z)\in \mc D_v$.
\end{theorem}

\begin{proof}
($\Leftarrow$) It was established in Lemma \ref{core-d-related} that each word is $\mc D$-related to its core.
Thus, it suffices to show that if $u_1\mathcal{D}u_2$ in $\mathcal{O}_n^\alpha$ and $v_1\mathcal{D}v_2$ in $\mathcal{O}_n^\beta$, then $u_1v_1\mathcal{D}u_2v_2$. We assume $u_1,u_2,v_1,v_2\neq 1$. Otherwise, if $u_1=1$, then $u_2$ must also be $1$, as in $\mathcal{J}_n$, $1$ is $\mathcal{D}$-related only to itself. Then $v_1\mathcal{D}v_2$.

There exist $x_1,\dots,x_4 \in \mathcal{O}_n^\alpha$ and $y_1,\dots,y_4 \in \mathcal{O}_n^\beta$ such that $x_1u_1 x_2=u_2$, $x_3 u_2 x_4=u_1$, $y_1 v_1 y_2=v_2$, $y_3 v_2 y_4=v_1$. Let $i$ be the index of the first letter of $u_2$ and $j$ be the index of the last letter of $v_2$. We then have that $\alpha_ix_1y_1u_1v_1x_2y_2\beta_j\to\alpha_ix_1u_1x_2y_1v_1y_2\beta_j\to\alpha_iu_2v_2\beta_j\to u_2v_2$ from contextual commutation (Lemma ~\ref{lemma:general}) and idempotence. The $\alpha_i$ and $\beta_i$ are added to ensure that contextual commutation may be used, as $x_1y_1=1$ or $x_2y_2=1$ may be the case. By switching the appropriate indices and $x$'s and $y$'s, one may obtain $u_1v_1$ from right and left multiplication on $u_2v_2$ in an identical fashion.
Then if $w\in\mathcal{O}_n$ has core $xy$ for $x\in\mathcal{O}_n^{\alpha}, y\in\mathcal{O}_n^{\beta}$ and $x\mathcal{D}u$, $y\mathcal{D}v$, we have that $w,xy\in\mathcal{D}_{uv}$. 

($\Rightarrow$) 
We observe that $z_1=z_2$ in $\mc O_n$ implies that there is a sequence of rewriting rules among (1)--(5), (1a), (2a), and (3a) that rewrites $z_1$ into $z_2$ by Corollary~\ref{cor:presentation}. Rules (1), (2), (3), and (5) are rules within $\mc O_n^\gamma$ for $\gamma\in \{\alpha,\beta\}$.
We observe that if $w_1\rar w_2$ is one of the rules (4), (1a), (2a), and (3a), then $p_\gamma(w_1)\rar p_\gamma(w_2)$ is one of the rules $\gamma_i\rar \gamma_i$, (1), (2), (3) in $\mc O_n^\gamma$, respectively. Therefore, $z_1=z_2$ implies that $p_\gamma(z_1)=p_\gamma(z_2)$.   
Let $w\in \mc O_n$ and $w\, \mc D\, uv$ for $u\in \mc O_n^\alpha$ and $v\in \mc O_n^\beta$. 
By Lemma~\ref{core-d-related}, $p(w)=p_\alpha(w)p_\beta(w)\,\mc D \,uv$. 
There are $z_1,\ldots,z_4$ such that  $z_1p_\alpha(w)p_\beta(w)z_2=uv$ and $p_\alpha(w)p_\beta(w)=z_3uvz_4$.
We have $p_\alpha(z_1wz_2)=p_\alpha(z_1)p_{\alpha}(w)p_\alpha(z_2)=u$ and 
$p_\beta(z_1wz_2)=p_\beta(z_1)p_\beta(w)p_\beta(z_2)=v$. Similarly, 
$p_\alpha(w)=p_\alpha (z_3)up_\alpha(z_4)$ and 
$p_\beta(z_3)vp_\beta(z_4)=p_\beta(w)$. Therefore $p_\alpha(z)\in \mc D_u$ and $p_\beta(z)\in \mc D_v$.

\end{proof}

One may see from Theorem~\ref{thm:main} that  
the $\mathcal{D}$-classes of $\mathcal{O}_n$ are \textit{exactly} such $\mathcal{D}_{xy}$ as described with the projection map,
as each word in the monoid appears in exactly one $\mathcal{D}$-class.

\medskip
Directly from Theorem~\ref{thm:main}, we have the following corollary.

\medskip
\begin{corollary}
There is a bijection between the Green's classes of $\mathcal{O}_n$ and those of $\mathcal{J}_n\times\mathcal{J}_n$.
\end{corollary}

\medskip

\begin{figure}[h]
    \centering
    \begin{tikzpicture}[vertex/.style={circle, draw=black!60, fill=gray!5, very thick, minimum size=1.5cm, inner sep=2mm}]

\node 
(a0) at (1,1.5) {[$\varepsilon$]};
\node 
(a1) at (0,0) {[$\alpha_1$]};
\node
(b1) at (2,0) {[$\beta_1$]};
\node
(a1b1) at (1,-1.5) {[$\alpha_1\beta_1$]};
\node
(a1a3) at (-1,-1.5) {[$\alpha_1\alpha_3$]};
\node
(b1b3) at (3,-1.5) {[$\beta_1\beta_3$]};
\node
(a1a3b1) at (0,-3) {[$\alpha_1\alpha_3\beta_1$]};
\node
(a1b1b3) at (2,-3) {[$\alpha_1\beta_1\beta_3$]};
\node
(a1a3b1b3) at (1,-4.5) {[$\alpha_1\alpha_3\beta_1\beta_3$]};

\draw[->, ultra thick] (a0) -- (a1);
\draw[->, ultra thick] (a0) -- (b1);
\draw[->, ultra thick] (a1) -- (a1a3);
\draw[->, ultra thick] (a1) -- (a1b1);
\draw[->, ultra thick] (b1) -- (a1b1);
\draw[->, ultra thick] (b1) -- (b1b3);
\draw[->, ultra thick] (a1a3) -- (a1a3b1);
\draw[->, ultra thick] (a1b1) -- (a1a3b1);
\draw[->, ultra thick] (a1b1) -- (a1b1b3);
\draw[->, ultra thick] (b1b3) -- (a1b1b3);
\draw[->, ultra thick] (a1a3b1) -- (a1a3b1b3);
\draw[->, ultra thick] (a1b1b3) -- (a1a3b1b3);
    
\end{tikzpicture}
\caption{$\mathcal{D}$-classes of $\mathcal{O}_5$. A node labeled $[w]$ denotes the $\mathcal{D}$-class of $w$. An arrow from $w$ to $v$ denotes that $v\leq_{\mathcal{D}}w$. The $\mathcal{D}$-classes of $\mathcal{O}_n^{\alpha}$ and $\mathcal{O}_n^{\beta}$ are along the upper flanks of the diamond.}
\label{fig:o5_dclasses}\vskip -3mm
\end{figure}
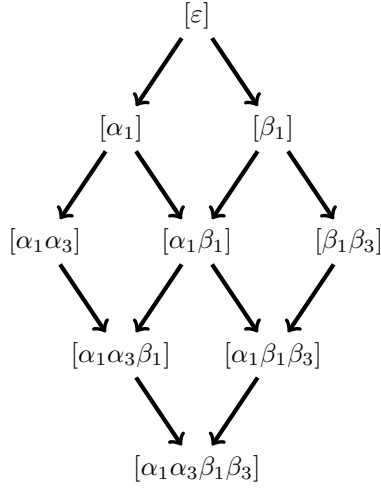

Notice that for $\mathcal{O}_n^\gamma\simeq\mathcal{J}_n$  ($\gamma \in\{ \alpha, \beta\}$), since every (non-identity) word is $\mathcal{D}$-related to a word of the form $\gamma_1\gamma_3\dots\gamma_{2k-1}$, where  $0<k\leq\lfloor\frac{n}{2}\rfloor$, we find that in $\mathcal{O}_n$, such is true for words of the form $\alpha_1\alpha_3\dots\alpha_{2k_1-1}\beta_1\beta_3\dots\beta_{2k_2-1}$ where $0<k_1,k_2\leq \lfloor\frac{n}{2}\rfloor$. This then gives the diagram of the Green's classes of $\mathcal{O}_n$ a diamond shape, pictured for $\mathcal{O}_5$ in Figure \ref{fig:o5_dclasses}. The figure for $\mathcal{O}_6$ looks identical to that of $\mathcal{O}_5$ and that for $\mathcal{O}_7$ would have seven more classes to extend the $3\times3$ diamond to a $4\times4$ diamond. The diagrams for $\mathcal{O}_n$ where $n$ is odd look identical to their $n-1$ counterparts due to the fact that increasing diagrams from an even $n-1$ to an odd $n$ strings cannot allow an additional cup on the top of any one word's diagram.

\section{Concluding remarks}

With this paper, we propose normal forms for elements of the origami monoids $\mc O_n$ 
that rely on 
normal forms for elements of the Jones monoid.
In search of normal forms, we discovered new relations resulting from the presentation inspired by DNA origami structures, such as contextual commutation (Lemma~\ref{lemma:general}). 
Contextual commutation
relations between two types of generators are used to separate the generator types  as much as possible, such that the  Jones normal forms can be used for each generator type in the description of the normal forms. 

We showed 
a relationship of  the Green's  $\mathscr{D}$-classes of $\mathcal{O}_n$ and $\mathcal{J}_n$ under the projection $p$, as was hinted in~\cite{GarrettJKS19}. 
Specifically, we showed that the $\mathscr{D}$-classes of $\mathcal{O}_n$ are in one-to-one correspondence with the $\mathscr{D}$-classes of $\mathcal{J}_n \times \mathcal{J}_n$.

The definition of origami  monoids 
by presentation is based on the presentations of Jones monoids. The set of generators of $\mathcal{J}_n$ is doubled into two types, where the submonoid generated by each type is isomorphic to $\mathcal{J}_n$; partial commutation among differing types of generators is allowed; and substitution-type rules inspired by relations from $\mathcal{J}_n$ are added. While these relations occurred naturally from the examination of DNA origami structures, the construction of doubling generators and imposing substitution relations can be generalized to other presentations of algebraic structures. More than two types of generators may also be used.
How such a general approach of expanding algebraic systems affects properties of the new constructs is an interesting algebraic question. 
It is also notable that contextual commutation appears in this construction, despite not being apparent from the defining relations. This phenomenon may generalize to other monoids under appropriate doubling constructions as well. Identifying conditions that induce such contextual commutation may be of interest, because separating the generators within a word may be a way to study larger classes of algebraic systems. 

\section*{Acknowledgment}
This work is partially supported by grants NSF DMS-2054321, CCF-2107267, and the W.M. Keck Foundation.

\bibliography{dna-new}

\end{document}